\documentclass[12pt]{amsart}

\usepackage{cite}
\usepackage[hypertexnames=false, colorlinks, linkcolor=blue, allcolors=blue]{hyperref}
\hypersetup{nesting=true,debug=true,naturalnames=true}
\usepackage[margin=1in]{geometry}
\usepackage{graphicx}
\usepackage{enumitem}
\usepackage{thm-restate}
\usepackage{cleveref}
\usepackage{overpic}
\usepackage{booktabs}

\usepackage[colorinlistoftodos]{todonotes}

\newtheorem{theorem}{Theorem}[section]
\newtheorem{lemma}[theorem]{Lemma}
\newtheorem{conjecture}[theorem]{Conjecture}

\newtheorem{definition}[theorem]{Definition}
\newtheorem{proposition}[theorem]{Proposition}
\newtheorem{corollary}[theorem]{Corollary}
\newtheorem{question}[theorem]{Question}

\theoremstyle{definition}
\newtheorem{example}[theorem]{Example}
\newtheorem{remark}[theorem]{Remark}

\DeclareMathOperator{\MCG}{MCG}
\DeclareMathOperator{\LMCG}{LMCG}
\DeclareMathOperator{\SMCG}{SMCG}
\DeclareMathOperator{\PMCG}{PMCG}
\DeclareMathOperator{\LPMCG}{LPMCG}
\DeclareMathOperator{\len}{len}

\begin{document}

\title{Equivariant unknotting numbers of strongly invertible knots} 

\author{Keegan Boyle}  
\address{Department of Mathematical Sciences, New Mexico State University, USA} 
\email{kboyle@nmsu.edu}

\author{Wenzhao Chen}
\address{Institute of Mathematical Sciences, ShanghaiTech University, China}
\email{chenwzh@shanghaitech.edu.cn}

\setcounter{section}{0}

\begin{abstract}
We study symmetric crossing change operations for strongly invertible knots. Our main theorem is that the most natural notion of equivariant unknotting number is not additive under connected sum, in contrast with the longstanding conjecture that unknotting number is additive. 
\end{abstract}
\maketitle
\tableofcontents
\section{Introduction}
The \emph{unknotting number} $u(K)$ of a knot $K$ is the minimum number of transverse self-intersections in a regular homotopy to the unknot. First defined in \cite{MR1545700}, the unknotting number has a long history but remains mysterious. For example, according to KnotInfo \cite{knotinfo}, the knots with unknown unknotting number and 10 or fewer crossings are 
\[
10_{11},10_{47},10_{51},10_{54},10_{61},10_{76},10_{77},10_{79}, \mbox{and } 10_{100}.
\]
Another important open question about the unknotting number concerns its additivity under connected sum of knots.

\begin{conjecture}[Additivity of unknotting number \cite{MR1545700}] \label{conj:additivity}
Let $K$ and $K'$ be knots in $S^3$. Then $u(K\#K') = u(K) + u(K')$.
\end{conjecture}
In light of the evident difficulty of these questions, a natural idea is to consider the unknotting number in the presence of additional structure. To this end, we study equivariant unknotting numbers for strongly invertible knots. 

A \emph{strongly invertible knot} is a knot $K \subset S^3$ along with a smooth symmetry which preserves the orientation on $S^3$ but reverses the orientation on $K$. As a consequence of geometrization, any such symmetry is conjugate in the diffeomorphism group of pairs $(S^3, K)$ to a $180^{\circ}$-rotation around an unknot which intersects $K$ in two points (see for example \cite{boyle2023classificationsymmetriesknots}). For an example, see any of the diagrams appearing in the first or third row of Figure \ref{fig:selfintersectiontypes}. 

Naturally, we would like to consider an equivariant version of Conjecture \ref{conj:additivity}. To make this precise, we define the \textbf{total equivariant unknotting number} (denoted by $\widetilde{u}(K)$; see Definition \ref{def:totaleu}) which is the minimum number of transverse self-intersections in a regular and equivariant homotopy to the unknot. In this setting, we can resolve the equivariant version of Conjecture \ref{conj:additivity}.

\begin{restatable}{theorem}{nonadditive}
\label{thm:nonadditive}
There are strongly invertible knots $K_1$ and $K_2$ and an equivariant connected sum $K_1\#K_2$ such that $\widetilde{u}(K_1\#K_2) > \widetilde{u}(K_1) + \widetilde{u}(K_2)$. In particular, the total equivariant unknotting number is not additive or even sub-additive.
\end{restatable}

Our approach to Theorem \ref{thm:nonadditive} is to consider the natural classification of equivariant transverse self-intersections into three types, which we call type A, type B, and type C. These self-intersections correspond to three types of equivariant crossing change to which we apply the same labels; see Figure \ref{fig:selfintersectiontypes}. As a stepping stone to Theorem \ref{thm:nonadditive}, we define, for $X \in \{A,B,C\}$, the \textbf{type $X$ unknotting number} $\widetilde{u}_X(K)$ of a strongly invertible knot $K$ as the minimum number of type $X$ self-intersections in a regular and equivariant homotopy to the unknot, where all self-intersections are required to be type $X$. For example in Figure \ref{fig:selfintersectiontypes}, the second column demonstrates that $\widetilde{u}_B(4_1) = 1$, the third column demonstrates that $\widetilde{u}_C(4_1) = 1$, and the first column demonstrates that $|\widetilde{u}_A(4_1) - \widetilde{u}_A(3_1)| \leq 1.$ In contrast with the total equivariant unknotting number, it may a priori be the case that a knot cannot be unknotted with type $X$ moves, in which case we say that $\widetilde{u}_X(K) = \infty$. 

\begin{figure}
\begin{overpic}[width=300pt, grid=false]{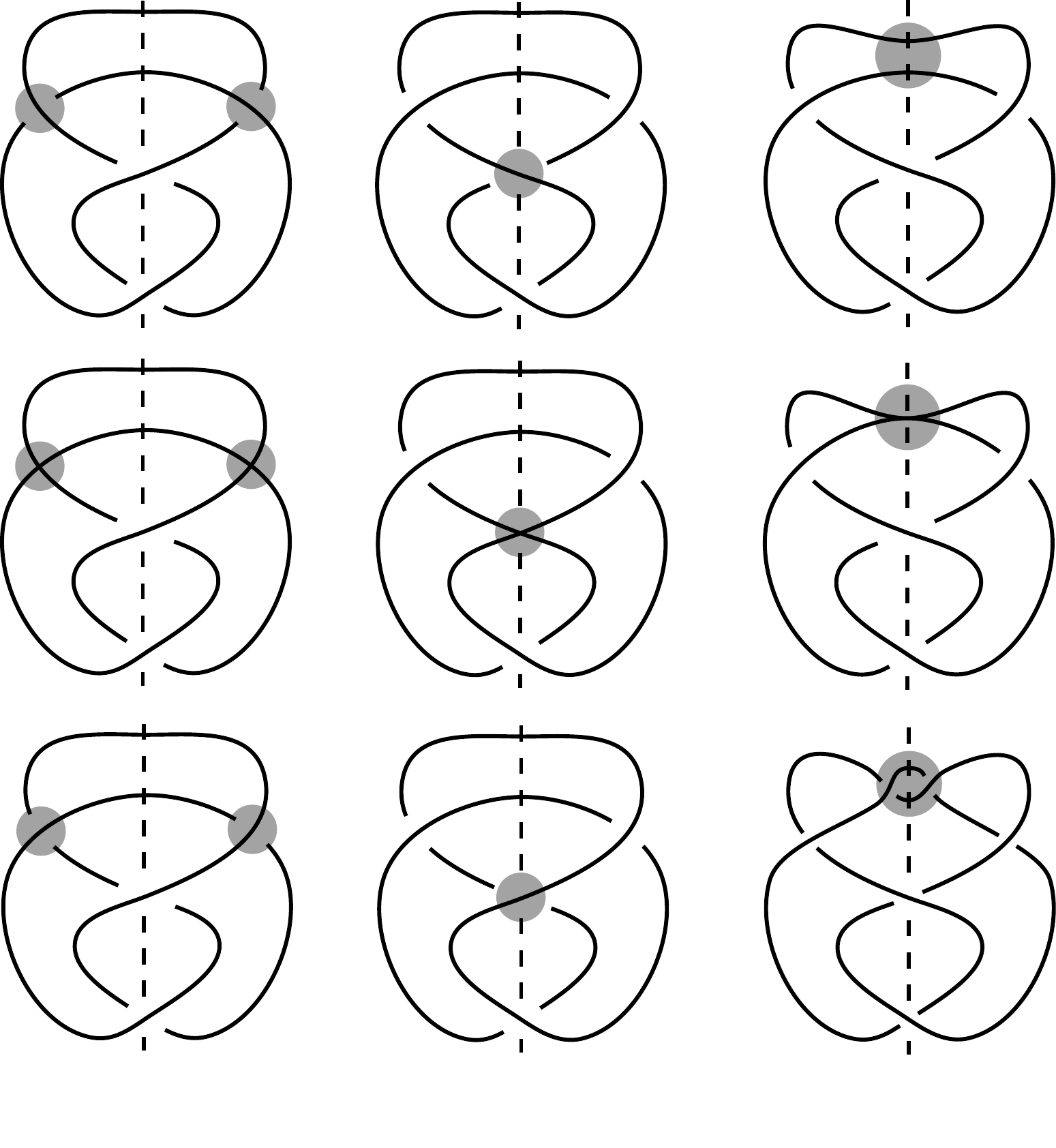}
    \put (7.5, 3) {Type A}
    \put (40, 3) {Type B}
    \put (74, 3) {Type C}
  \end{overpic}
\caption{Examples of the three types of equivariant transverse self-intersections, realized as a movie starting on the figure-eight knot. The first column is Type A (`off the axis'), the middle column is Type B (`through the axis'), and the right column is Type C (`along the axis').}
\label{fig:selfintersectiontypes}
\end{figure}

The majority of this paper is concerned with studying these restricted notions of equivariant unknotting number, which will culminate in Theorem \ref{thm:nonadditive}. In fact, we will see that Theorem \ref{thm:nonadditive} relies on the non-additivity of the type C unknotting number.

\begin{restatable}{theorem}{typeCadditive}
\label{thm:typeCadditive}
Let $K$ by a strongly invertible knot with three non-trivial summands. Then $\widetilde{u}_C(K) = \infty$. In particular, the type C unknotting number is not additive under connected sum.
\end{restatable}

For type A and type B crossing changes, we do not know whether the corresponding equivariant unknotting numbers are additive.

\begin{conjecture}
Let $K\#K'$ be an equivariant connected sum of two strongly invertible knots $K$ and $K'$. Then $\widetilde{u}_A(K\#K') = \widetilde{u}_A(K) + \widetilde{u}_A(K')$ and $\widetilde{u}_B(K\#K') = \widetilde{u}_B(K) + \widetilde{u}_B(K')$.
\end{conjecture}

In building towards Theorem \ref{thm:nonadditive}, we also provide some answers to elementary questions about $\widetilde{u}_A, \widetilde{u}_B,$ and $\widetilde{u}_C$, which we state in Sections \ref{subsec:uo} and \ref{subsec:lb}.

\subsection{Unknotting operations} \label{subsec:uo} Which types of equivariant crossing changes are unknotting operations? In other words, is $\widetilde{u}_X(K) < \infty$? For the type A unknotting number, we have the following. 

\begin{restatable}{theorem}{typeAunknots}
\label{thm:typeAunknots}
For any strongly invertible knot $K$, we have $\widetilde{u}_A(K) < \infty$.
\end{restatable}

On the other hand, we do not know if every strongly invertible knot can be unknotted with type B moves.

\begin{restatable}{question}{typeBunknots}
\label{q:typeBunknots}
For any strongly invertible knot $K$, is it true that $\widetilde{u}_B(K) < \infty$?
\end{restatable}

\begin{remark} \label{rmk:3131}
Consider the connected sum of a trefoil with its reverse $K = 3_1 \# r3_1$ with the symmetry that exchanges the two summands. Any minimal crossing number diagram for $K$ does not have any on-axis crossings. Thus we find it surprising that $\widetilde{u}_B(K) \neq \infty$. Indeed, we can see in Figure \ref{fig:3131} that $\widetilde{u}_B(K) \leq 4$.
\end{remark}

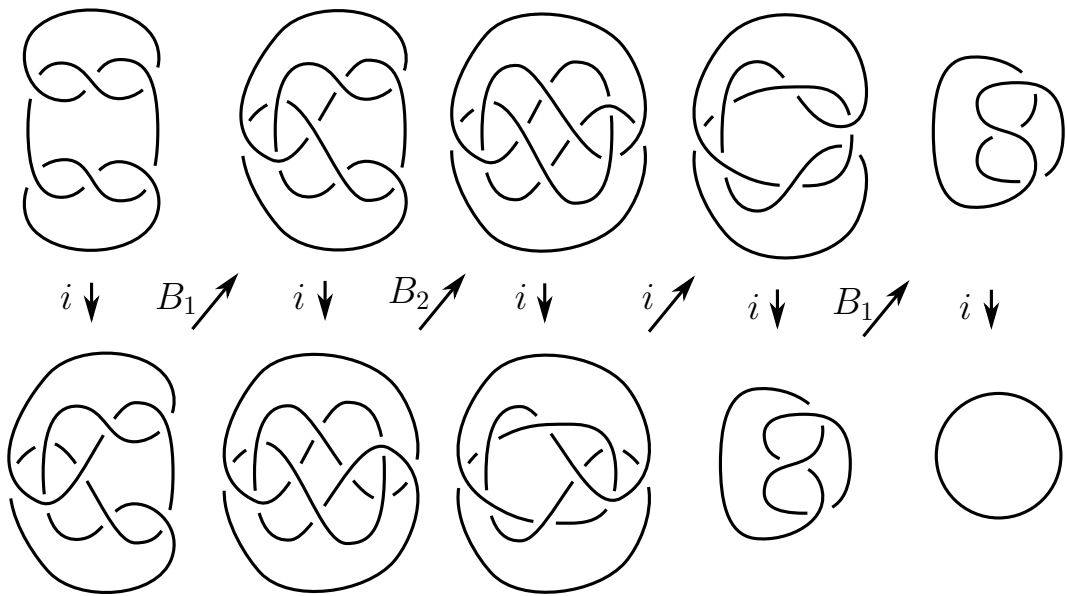
\begin{figure}
\begin{overpic}[width=400pt, grid=false]{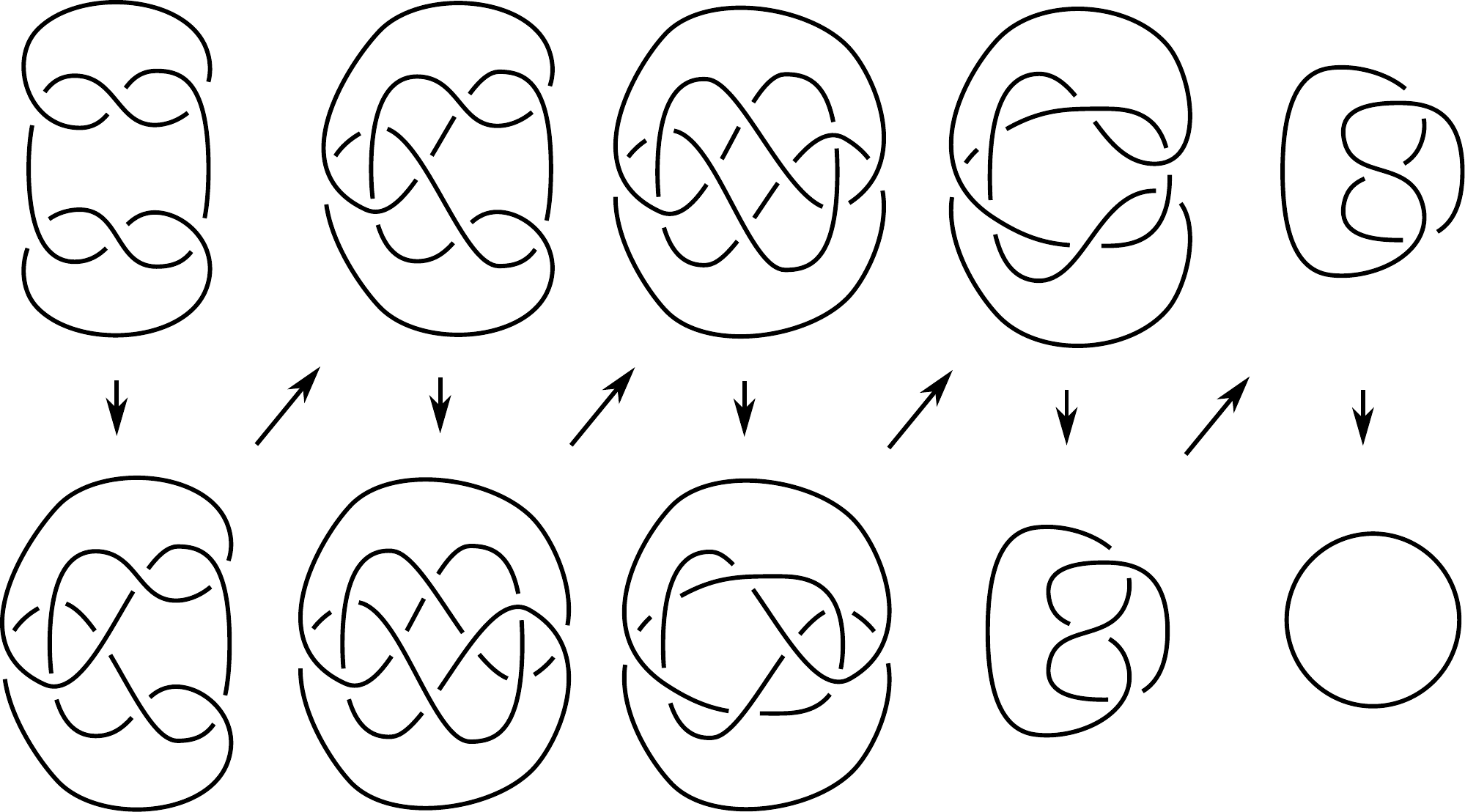}
\put (5, 27) {\large$i$}
\put (27, 27) {\large$i$}
\put (48, 27) {\large$i$}
\put (70, 26) {\large$i$}
\put (90, 26) {\large$i$}

\put (14, 27) {\large$B_1$}
\put (36, 27.5) {\large$B_2$}
\put (60, 27) {\large$i$}
\put (78, 26.5) {\large$B_1$}
\end{overpic}
\caption{An unknotting sequence of $3_1 \# r3_1$ consisting of equivariant isotopies indicated by arrows labelled with an $i$, and type B moves indicated by arrows labelled with a B$_x$, where $x$ is the number of type B moves applied. For compactness, the axis of symmetry is horizontal in each diagram. A total of 4 type B moves are used so that $\widetilde{u}_B(3_1\#r3_1) \leq 4$.}
\label{fig:3131}
\end{figure}

Another interesting observation about Question \ref{q:typeBunknots} is that a positive answer would imply a positive answer to Nakanishi's 4-move conjecture \cite[Conjecture B]{MR0881755} (this problem also appears on the Kirby problem list \cite[1.59(3)(a)]{kirbylist}); see Corollary \ref{cor:typeB4move}.

Finally, we give a complete classification of strongly invertible knots which can be unknotted with type C moves. In the following theorem, a $(1,2)$-knot refers to a genus one 2-bridge knot. That is a knot which can be decomposed into a union of 4 arcs by the standard genus 1 Heegaard splitting of $S^3$, where each handlebody contains a pair of boundary-parallel arcs.

\begin{restatable}{theorem}{typeC}
\label{thm:1_2typeC}
A strongly invertible knot $K$ has $\widetilde{u}_C(K) < \infty$ if and only if $K$ is a $(1,2)$-knot such that the axis of symmetry is the core of one of the handlebodies in the $(1,2)$ decomposition.
\end{restatable}

The proof of Theorem \ref{thm:1_2typeC} involves studying the symmetric mapping class group of a solid torus with 4 marked points on the boundary; we give a list of generators for this mapping class group in Proposition \ref{prop:smcggens}.

\subsection{Lower bounds} \label{subsec:lb} We now state some lower bounds for $\widetilde{u}_A(K),\widetilde{u}_B(K),$ and $\widetilde{u}_C(K)$ which are useful in proving Theorem \ref{thm:nonadditive}, but may be of independent interest. Our theorems are in terms of the quotient knots $\mathfrak{q}_1(K)$ and $\mathfrak{q}_2(K)$ of a strongly invertible knot $K$; see Definition \ref{def:quotients}.

\begin{restatable}{theorem}{typeAbounds}
\label{thm:typeAbounds}
Let $K$ be a strongly invertible knot. Then $\widetilde{u}_A(K) \geq \max(u(\mathfrak{q}_1(K)), u(\mathfrak{q}_2(K))).$
\end{restatable}

To state our lower bound for $\widetilde{u}_B(K)$, let $u_4(K)$ be the minimum number of 4-moves needed to unknot $K$; see Section \ref{sec:B}.

\begin{restatable}{theorem}{typeBbounds}
\label{thm:typeBbounds}
Let $K$ be a strongly invertible knot. Then $\widetilde{u}_B(K) \geq u_4(\mathfrak{q}_1(K)) + u_4(\mathfrak{q}_2(K))$.
\end{restatable}

To state our lower bound for $\widetilde{u}_C(K)$, let $u_{nb}(K)$ be the minimum number of non-orientable band moves needed to unknot $K$; see Section \ref{sec:C}.

\begin{restatable}{theorem}{typeCbounds}
\label{thm:typeCbounds}
Let $K$ be a strongly invertible knot. Then $\widetilde{u}_C(K) \geq u_{nb}(\mathfrak{q}_1(K)) + u_{nb}(\mathfrak{q}_2(K))$.
\end{restatable}

As a consequence of these theorems, we are able to show that there are knots for which $\widetilde{u}_A(K) - u(K)$ is arbitrarily large (see Corollary \ref{cor:typeAquot}), and that there is a knot for which $\widetilde{u}_B(K) - u(K) \geq 2$ (see Example \ref{ex:connectsumof41s}). On the other hand for type C moves, Theorem \ref{thm:typeCadditive} shows that $\widetilde{u}_C(K) - u(K)$ can be infinite. We make the following conjecture about type B moves.

\begin{conjecture}
There is a sequence $K_n$ of strongly invertible knots such that $\widetilde{u}_B(K_n) - u(K_n)$ is unbounded.
\end{conjecture}

\subsection{Torus knots}
Along with the elementary lower bounds in the previous section, we compute an upper bound on the total equivariant unknotting numbers for torus knots. This upper bound turns out to be sharp, as is recorded in the following theorem.

\begin{restatable}{theorem}{torusknots}
\label{thm:torusknots}
Let $K$ be the $(p,q)$-torus knot with its unique strong inversion. Then 
\[
    \widetilde{u}(K) = u(K) = \dfrac{(p-1)(q-1)}{2}.
\]
Furthermore, there exists a minimal length equivariant unknotting sequence consisting of only type A and type B moves.
\end{restatable}

The theorem that $u(K) = \dfrac{(p-1)(q-1)}{2}$, known as the Milnor conjecture and proved in \cite{MR1241873}, is closely related to ideas in algebraic geometry about links of plane curve singularities. Milnor's conjecture originated from observations in algebraic geometry, namely that a particular resolution of a plane curve singularity produces an algebraic surface in $B^4$ with boundary $T(p,q)$. Kronheimer and Mrowka then proved that this surface has the minimum genus possible with boundary $T(p,q)$, which along with the fact that the 4-genus is a lower bound on $u(K)$ proved that $u(K) \geq \dfrac{(p-1)(q-1)}{2}$. On the other hand, the upper bound on $u(K)$, has been known for longer; see \cite{MR699004}. 

For the equivariant unknotting number $\widetilde{u}(K)$, we also have a lower bound coming from the equivariant 4-genus; see Proposition \ref{prop:4genus}. Furthermore, the equivariant 4-genus is equal to the 4-genus, since the resolved algebraic surfaces can be seen to be invariant under complex conjugation, so we can readily check that $\widetilde{u}(K) \geq \dfrac{(p-1)(q-1)}{2}$. Hence the content of Theorem \ref{thm:torusknots} is the upper bound on $\widetilde{u}(K)$.

For this upper bound, one may be tempted to emulate \cite[Theorem 4]{MR1733329}, which proves that $u(K) \geq \dfrac{(p-1)(q-1)}{2}$. However, we were unable to prove a symmetric version of the lemma used in this proof. Instead, we emulate Rudolph's original diagrammatic proof \cite{MR699004} using the braid group, which is somewhat less concise than A'Campo's proof, and only becomes more complicated in the equivariant setting. We use a particular type of symmetric braid which we call an intravergent braid (see Definition \ref{def:intravergentbraid}). Our main contribution here is the following proposition, which may be of independent interest.

\begin{restatable}{proposition}{positivebraidprop}
  \label{prop:weight}
  Let $B$ be a positive intravergent braid with $s = 2n+1$ strands and length $\ell = \len(B)$ such that the closure of $B$ is a knot $\widehat{B}$. Then $\widetilde{u}(\widehat{B}) \leq \dfrac{\ell - s + 1}{2}$.
\end{restatable}

\subsection{Questions}
Finally, we record some basic unanswered questions here.

\begin{question}
In Figure \ref{fig:selfintersectiontypes}, we see that for the unique (up to symmetry) strong inversion on the figure-eight knot we have $\widetilde{u}_B(4_1) = \widetilde{u}_C(4_1) = 1$, and that $\widetilde{u}_A(4_1) \leq 2$, since $\widetilde{u}_A(3_1) = 1$ (by inspection). Is $\widetilde{u}_A(4_1) = 1$ or is $\widetilde{u}_A(4_1) = 2$? More generally, what are the type A unknotting numbers of the strongly invertible twist knots $K_n$ as shown in Figure \ref{fig:twistknots}?
\end{question}

\begin{remark}
For these twist knots $K_n$, Theorem \ref{thm:typeAbounds} gives the lower bound $\widetilde{u}_A(K_n) \geq n$, since the quotient (as shown in Figure \ref{fig:twistknots}) is $T(2,2n+1)$ which has unknotting number $n$, and the other quotient is the unknot. However, the best upper bound we found was $\widetilde{u}_A(K_n) \leq 2n$ coming from the apparent sequence of type A moves.
\end{remark}

\begin{question}
In Remark \ref{rmk:3131} we saw that $\widetilde{u}_B(3_1 \# r3_1) \leq 4$, but our best lower bound comes from Theorem \ref{thm:typeBbounds} which gives that $\widetilde{u}_B(3_1 \# 3_1) \geq 2$. What is the exact value of $\widetilde{u}_B(3_1 \# r3_1)$?
\end{question}

\subsection{Acknowledgements} We would like to thank Kenneth L. Baker and Maggie Miller for directing us to some relevant literature, and Ben Williams and Liam Watson for helpful conversations and advice. This project started while both authors were postdocs at UBC, and the second author was partially supported by the Pacific Institute for the Mathematical Sciences (PIMS).

\section{Equivariant crossing changes and unknotting numbers}
\label{sec:background}

Fix an order 2 symmetry $\rho\colon S^1 \to S^1$ with two fixed points, and an order 2 symmetry $\rho\colon S^3 \to S^3$ with a fixed circle. Consider an equivariant homotopy $h\colon (S^1 \times I,\rho) \to (S^3 \times I,\rho)$ between two strongly invertible knots $K = h(S^1,0)$ and $K' = h(S^1,1)$, such that $h(S^1, t) = (S^3, t)$ for all $t \in I$, and $h$ has only transverse self-intersections in the interior. We would like to define the equivariant unknotting number as the minimum number of self-intersection points in such a homotopy between a given strongly invertible knot and the unknot. The presence of a symmetry, however, means that these self-intersections are naturally classified into the following three types.
\begin{enumerate}
\item Self-intersections which do not lie on the axis of symmetry come in symmetric pairs. We refer to such self-intersections as \emph{type A}.
\item Self-intersections which do lie on the axis of symmetry come in two types.
\begin{enumerate}
\item If the self-intersection point is the image of a pair of points exchanged by the symmetry on $S^1$, then we call the self-intersection \emph{type B}.
\item If the self-intersection point is the image of the two fixed points on $S^1$, then we call the self-intersection point \emph{type C}.
\end{enumerate}
\end{enumerate}

These three types of self-intersections can be seen diagrammatically in Figure \ref{fig:selfintersectiontypes}, along with the corresponding crossing-change moves. Note that a type B or type C move consists of a single crossing change, but a type A move consists of a pair of crossing changes.

\begin{definition}
For X $\in \{A,B,C\}$, the \textbf{type X unknotting number} $\widetilde{u}_X(K)$ of a strongly invertible knot $K$ is the minimum number of type X moves necessary to reduce $K$ to the unknot. If $K$ cannot be reduced to the unknot in a finite number of type X moves, then we say that $\widetilde{u}_X(K) = \infty$.
\end{definition}

\begin{definition} \label{def:totaleu}
The \textbf{total equivariant unknotting number} $\widetilde{u}(K)$ is the minimum number of equivariant crossing changes of all types necessary to unknot $K$. 
\end{definition}
\begin{remark}
Noting that a type A move consists of two crossing changes, we have the immediate inequalities 
\begin{itemize}
\item $\widetilde{u}(K) \leq 2\widetilde{u}_A(K)$,
\item $\widetilde{u}(K) \leq \widetilde{u}_B(K)$, and
\item $\widetilde{u}(K) \leq \widetilde{u}_C(K)$.
\end{itemize}
We could also have chosen the total equivariant unknotting number to only increment by one for each type A move. We made the choice to increment by two for each type A move so that we have a stronger lower bound on the equivariant 4-genus.
\end{remark}

\begin{definition}
The \textbf{equivariant 4-genus} $\widetilde{g}_4(K)$ of a strongly invertible knot $K$ is the minimum genus of a surface $\Sigma$ smoothly properly embedded in $B^4$ such that there is a smooth extension $\rho$ of the symmetry on $S^3$ to $B^4$ with $\rho(\Sigma) = \Sigma$.
\end{definition}
\begin{proposition} \label{prop:4genus}
Let $K$ be a strongly invertible knot. Then $\widetilde{g}_4(K) \leq \widetilde{u}(K)$. 
\end{proposition}
\begin{proof}
As in the non-equivariant setting, we can construct a cobordism between two knots related by a sequence of equivariant crossing changes, where the genus of the cobordism is equal to the number of crossing changes. Indeed, it is easy to see that the standard cobordism can be made equivariant at a type A, B, or C move. (Note that this matches the definition of $\widetilde{u}(K)$ since a type A move contributes two crossing changes, while a type B or C move contributes one.)
\end{proof}

\begin{remark}
Note that for any two strongly invertible knots $K_1$ and $K_2$, we have that $\widetilde{g}_4(K_1 \# K_2) \leq \widetilde{g}_4(K_1) + \widetilde{g}_4(K_2)$, since the two minimal genus surfaces can be glued equivariantly. 
\end{remark}

\begin{definition} \label{def:quotients}
Let $K$ be a strongly invertible knot with axis of symmetry $A$. Then in the quotient of $S^3$ by the strong inversion, there is a quotient theta graph consisting of the image of $K$ and the image of $A$. The three arcs of
 this theta graph are the image $\overline{K}$ of $K$ and two arcs which make up the image of $A$ which we refer to arbitrarily as $h_1$ and $h_2$. The two \textbf{quotient knots} of $K$ are $\mathfrak{q}_1(K) = \overline{K} \cup h_1$ and $\mathfrak{q}_2(K) = \overline{K} \cup h_2$.
\end{definition}

\section{Torus knots}
In this section we will prove Theorem \ref{thm:torusknots}, which we restate here for convenience. 

\torusknots*

In order to prove Theorem \ref{thm:torusknots} we will work with symmetric braids. Although symmetric braids whose closures are strongly invertible knots have been studied in \cite{merz2024alexandermarkovtheoremsstrongly}, they study braids which are symmetric under rotation around a line in the plane of the diagram, while for our application it is more convenient to use braids which are symmetric under rotation around a line orthogonal to the plane of the diagram.
\begin{definition} \label{def:intravergentbraid} Let $\sigma_i$ refer to the positive half-twist between the $i$th and $i+1$th strands of a braid. An \emph{intravergent braid} is a braid on $2n+1$ strands given by a word of length $2m$ such that if $\sigma_i^{\pm1}$ appears at index $j$, then $\sigma_{2n+1-i}^{\pm1}$ appears at index $m - j$. (See for example Figure \ref{fig:intravergentbraid}.)
\end{definition}

\begin{figure}
\scalebox{.5}{\includegraphics{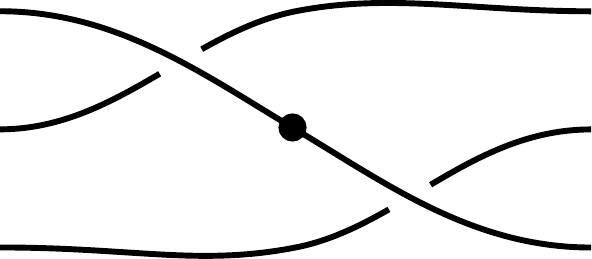}}
\caption{An intravergent braid on 3 strands given by the word $\sigma_1 \sigma_2$ when read left to right. The symmetry is given by $\pi$ rotation around the central marked point.}
\label{fig:intravergentbraid}
\end{figure}

Observe that the closure of an intravergent braid, if it is a knot, is a strongly invertible knot. Additionally, we will say that an intravergent braid is positive if every generator $\sigma_i$ appears with only positive powers. 

\begin{definition}
The \emph{length} of an intravergent braid $B$, written $\len(B)$, is the length of the underlying braid word for $B$. 
\end{definition}

The following proposition is an equivariant analog of \cite[Proposition on page 34]{MR699004}. This is our contribution needed to prove Theorem \ref{thm:torusknots}.

\positivebraidprop*

To prove this proposition we will need the following lemma.

\begin{lemma}\label{lemma:sigmaplow}
Let $w$ be a word in the (positive) generators $\{\sigma_i : 1 \leq i \leq 2n\}$ of the braid group $B_{2n+1}$ which begins with $\sigma_1$ and contains no other $\sigma_1$s. Then $w$ is equivalent (as an element of $B_{2n+1}$) to a word $w'$ with $\len(w') = \len(w)$ such that at least one of the following is true:
\begin{enumerate}
\item $w'$ is the word $\sigma_1 \sigma_2 \dots \sigma_i$ for some $i \leq 2n$, or
\item $w'$ is a word containing $\sigma_i^2$ for some $1 < i \leq 2n$, or 
\item $w'$ is a word containing exactly one $\sigma_1$, but not beginning with $\sigma_1$.
\end{enumerate}
Similarly, if $w$ ends with $\sigma_1$ and contains no other $\sigma_1$s, then $w$ is equivalent to a word $w'$ as above, but with the order of the letters in $w'$ reversed.

\end{lemma}

\begin{proof}
We proceed by induction on the length of $w$. For the base case, we have $w = \sigma_1$, which satisfies (1).

Now suppose that the statement is true for all words with length $< m$, and let $w$ be a word of length $m$ which starts with $\sigma_1$ and contains no other $\sigma_1$s. If $w$ is the word $\overline{w} = \sigma_1 \sigma_2 \dots \sigma_m$, then $w$ satisfies (1). If not, then there is a first letter in $w$ which differs from $\overline{w}$; call it $\sigma_i$ at index $j$ with $j \neq i$ and $i > 1$ so that $w$ begins with $\sigma_1 \sigma_2 \dots \sigma_{j-2} \sigma_{j-1} \sigma_i$. There are then three cases.
\begin{enumerate}[label=(\alph*)]
\item If $i = j-1$, then $w$ satisfies (2).
\item If $i > j$, then we can commute $\sigma_i$ past all of the previous letters in $w$, so that $w'$ begins with $\sigma_i \sigma_1 \sigma_2 \dots$ and hence $w$ satisfies (3).
\item If $i < j-1$, then we can commute $\sigma_i$ to the left until we run into $\sigma_{i+1}$, so that $w$ is equivalent to a word beginning with $\sigma_1 \sigma_2 \dots \sigma_i \sigma_{i+1} \sigma_i$. We can then apply the braid relation to get an equivalent word beginning with $\sigma_1 \sigma_2 \dots \sigma_{i-1}\sigma_{i+1} \sigma_{i} \sigma_{i+1}$. Now since $i > 1$, we can commute the first $\sigma_{i+1}$ to the beginning of the word so that we have a word $w'$ which begins $\sigma_{i+1} \sigma_1 \sigma_2 \dots \sigma_{i-1}\sigma_{i} \sigma_{i+1}$. Hence $w$ satisfies (3).
\end{enumerate}
\end{proof}

\begin{proof}[Proof of Proposition \ref{prop:weight}]
We will induct on the pair $(s,\ell)$ with respect to the lexicographical order. Clearly when $s = 1$ we have $\ell = 0$ and hence $\widehat{B}$ is the unknot. Thus the statement is true for $(s,\ell) = (1,0)$.

We now suppose that the statement is true for all parameters less than $(s,\ell)$, and consider an arbitrary braid $B$ with $s = 2n+1$ strands and length $\ell = 2m$ whose closure is a knot. We have four cases:
\begin{enumerate}
\item $\sigma_1$ does not appear as a letter in $B$,
\item $\sigma_1$ appears exactly once as a letter in $B$,
\item $\sigma_1$ appears exactly twice as a letter in $B$, or
\item $\sigma_1$ appears more than twice as a letter in $B$.
\end{enumerate} 

In case (1), either $\widehat{B}$ is a link, or $s = 1$ which was already covered as the base case.

In case (2), we can apply a pair of symmetric Reidemeister 1 moves, realized on the braid as a symmetric pair of Markov moves, which reduces the number of strands to $s - 2$ and the length of the word to $\ell - 2$. Thus by our inductive assumption we have that 
\[
    \widetilde{u}(\widehat{B}) \leq \dfrac{(\ell - 2) - (s-2) + 1}{2} = \dfrac{\ell - s + 1}{2}.
\]
In case (3) we have two sub-cases.
\begin{enumerate}
\item[(3.1)] The first case is when both appearances of $\sigma_1$ occur with index $\leq m$ or $> m$. In this case we consider the word $w$ which begins with the first occurrence of $\sigma_1$ in $B$ and ends with the letter preceding the second occurrence of $\sigma_1$. Let $x = \len(w)$. By Lemma \ref{lemma:sigmaplow} there are three possibilities.
\begin{enumerate}
\item[(3.1.1)] If $w$ is equivalent to the word $w' = \sigma_1 \sigma_2 \dots \sigma_x$, then observe that $w'\sigma_1$ is equivalent to $\sigma_1 \sigma_2 \sigma_1 \sigma_3 \sigma_4 \dots \sigma_x$ by commuting the trailing $\sigma_1$ past all the letters excepts $\sigma_1 \sigma_2$. Now we can apply the braid relation $\sigma_1 \sigma_2 \sigma_1 = \sigma_2 \sigma_1 \sigma_2$ to get an equivalent word $\sigma_2 \sigma_1 \sigma_2 \sigma_3 \sigma_4 \dots \sigma_x$. Note that all of these alterations to $w$ can be symmetrically applied to the other half of $B$ so that $B$ is equivalent to an intravergent braid $B'$ containing only a single $\sigma_1$. This reduces us to case (2).

\item[(3.1.2)] If $w$ is equivalent to a word $w'$ which contains $\sigma_i^2$ for some $i$, then applying a crossing change will delete $\sigma_i^2$ from $w'$. We can apply these alterations symmetrically in $B$ as a type A move to obtain a positive intravergent braid $B'$ with length $\ell - 4$ whose closure is a knot $\widehat{B'}$. By induction, we have that $\widetilde{u}(\widehat{B'}) \leq \dfrac{\ell - 4 - s + 1}{2}$ and hence
\[
\widetilde{u}(\widehat{B}) \leq 2+ \widetilde{u}(\widehat{B'}) \leq \dfrac{\ell -s + 1}{2}.
\]
\item[(3.1.3)] If $w$ is equivalent to a word $w'$ which still contains a single $\sigma_1$ but does not begin with $\sigma_1$, then by induction on the number of letters between the two occurrences of $\sigma_1$ we can reduce to case (3.1.1) or (3.1.2).
\end{enumerate}
\item[(3.2)] The second case is that $\sigma_1$ appears once with index $\leq m$ and once with index $> m$. We begin by considering the word beginning at the first occurrence of $\sigma_1$ and ending at index $m$. This word contains no other instances of $\sigma_1$, so we may apply Lemma \ref{lemma:sigmaplow} and alter $B$ symmetrically in the indices $> m$. If the result contains the square of a generator, then we may proceed as in case (3.1.2). Otherwise, the result is an equivalent intravergent braid $B'$ which still contains a single $\sigma_1$ with index $\leq m$ and a single $\sigma_1$ with index $> m$, and for which the subword from the first $\sigma_1$ to the middle index $m$ is $w = \sigma_1 \sigma_2 \dots \sigma_{i}$. In particular $i = \len(w)$. Note that by the symmetry of $B'$, the length $i$ word which occurs directly after $w$ in $B'$ is $\overline{w} = \sigma_{2n-i+1}\sigma_{2n-i+2} \dots \sigma_{2n}$. We now consider 3 possibilities depending on $i$ and $n$.
\begin{enumerate}
\item[(3.2.1)] If $i < n$, then $w$ commutes with $\overline{w}$ so that we can alter $B'$ by replacing $w\overline{w}$ with $\overline{w}w$. Furthermore, it is not hard to see that this replacement is realized by a symmetric isotopy of $B'$. The result is an intravergent braid with two occurrences of $\sigma_1$, both of which occur with index $> m$. This reduces us back to case (3.1).

\item[(3.2.2)] If $i = n$, then $\overline{w} = \sigma_{n+1} \sigma_{n+2} \dots \sigma_{2n}$, and we consider the word $w_{RHS}$ beginning with the first letter after $\overline{w}$ and ending with the $\sigma_1$ which has index $>m$. We may then apply Lemma \ref{lemma:sigmaplow} until $w_{RHS}$ either contains the square of a positive generator, in which case we proceed as in case (3.1.2), or is the word $\sigma_j \sigma_{j-1} \dots \sigma_2 \sigma_1$ for some $j \geq 1$. In the latter situation we have 3 cases.

\begin{enumerate}
\item[(3.2.2.1)] If $j < n$, we first commute $w_{RHS}$ past $\overline{w}$, and then we are in a situation similar to case (3.2.1), but with the order of the braid reversed. Hence a symmetric argument applies to reduce to case (3.1). 
\item[(3.2.2.2)] If $j = n$, then we consider the leading $\sigma_{n}$ of $w_{RHS}$, which we may commute past every letter of $\overline{w}$ except $\sigma_{n+1}$. Symmetrically, we also commute a $\sigma_{n+1}$ past every letter of $w$ except the trailing $\sigma_n$. The result is an intravergent braid with the middle four letters $\sigma_{n+1}\sigma_n\sigma_{n+1}\sigma_{n}$. This is the situation shown on the top left in Figure \ref{fig:typeBbraidmove}. Applying a type B move then produces an intravergent braid $B''$ which has $\sigma_{n+1}\sigma_n\sigma_{n+1}\sigma_{n}$ replaced with $\sigma_n\sigma_{n+1}$ as shown on the top right in Figure \ref{fig:typeBbraidmove}. We then calculate
\[
 \widetilde{u}(\widehat{B}) = \widetilde{u}(\widehat{B'}) \leq 1+ \widetilde{u}(\widehat{B''}) \leq 1 + \dfrac{(\ell-2) -s + 1}{2} = \dfrac{\ell -s + 1}{2}.   
\]
\item[(3.2.2.3)] If $j > n$, we will use a reductive argument to reduce to case (3.2.2.2). We begin by looking at the first letter of $w_{RHS}$ and $\overline{w}$. We have $\overline{w}\sigma_j = \sigma_{n+1} \sigma_{n+2} \dots \sigma_{2n} \sigma_j$, and commuting $\sigma_j$ as far to the left as possible, we have $\sigma_{n+1} \sigma_{n+2} \dots \sigma_{j}\sigma_{j+1}\sigma_j \sigma_{j+2} \dots \sigma_{2n}$. We then apply the braid relation to $\sigma_j \sigma_{j+1} \sigma_j$ to obtain $\sigma_{n+1} \sigma_{n+2} \dots \sigma_{j+1}\sigma_{j}\sigma_{j+1} \sigma_{j+2} \dots \sigma_{2n}$. We now commute the leftmost $\sigma_{j+1}$ all the way to the left to obtain $\sigma_{j+1} \sigma_{n+1} \sigma_{n+2} \dots \sigma_{2n}$. Applying all of these changes symmetrically to $B'$ produces a new intravergent braid $B''$ whose symmetry exchanges $\sigma_{j+1}$ with the preceding letter, which must therefore be $\sigma_{2n-j}$. We can now symmetrically commute $\sigma_{j+1}$ with $\sigma_{2n-j}$ so that the right half of $B''$ begins $\sigma_{2n-j}\sigma_{n+1} \sigma_{n+2} \dots \sigma_{2n}$. Since $j > n$, we have that $2n-j < n$, and hence we may commute $\sigma_{2n-j}$ to the right of this word. In total we have now replaced $\overline{w}w_{RHS}$ with $\overline{w}\sigma_{2n-j} \sigma_{j-1} \sigma_{j-2} \dots \sigma_1$. We can now apply the argument in part (c) of the proof of Lemma \ref{lemma:sigmaplow} to the word $\sigma_{2n-j} \sigma_{j-1} \sigma_{j-2} \dots \sigma_1$ to get $\sigma_{j-1} \sigma_{j-2} \dots \sigma_1 \sigma_{2n-j+1}$. We may now repeat this entire process with $w_{RHS} = \sigma_{j-1} \sigma_{j-2} \dots \sigma_1$, that is we have reduced $j$ by 1. We continue until $j = n$, which is case (3.2.2.2).
\end{enumerate}

\item[(3.2.3)] If $i = n+1$, then the middle four letters of $B'$, consisting of the last two letters of $w$ and the first two letters of $\overline{w}$, are $\sigma_n \sigma_{n+1} \sigma_n \sigma_{n+1}$. In this case we can perform a type B move to $B'$, replacing these four letters with $\sigma_{n+1}\sigma_n$ to obtain a new intravergent braid $B''$ with $\len(B'') = \len(B') - 2$. See Figure \ref{fig:typeBbraidmove}. By the inductive assumption, we then calculate 
\[
\widetilde{u}(\widehat{B}) = \widetilde{u}(\widehat{B'}) \leq 1+ \widetilde{u}(\widehat{B''}) \leq 1 + \dfrac{(\ell-2) -s + 1}{2} = \dfrac{\ell -s + 1}{2}.
\]

\begin{figure}
\scalebox{.3}{\includegraphics{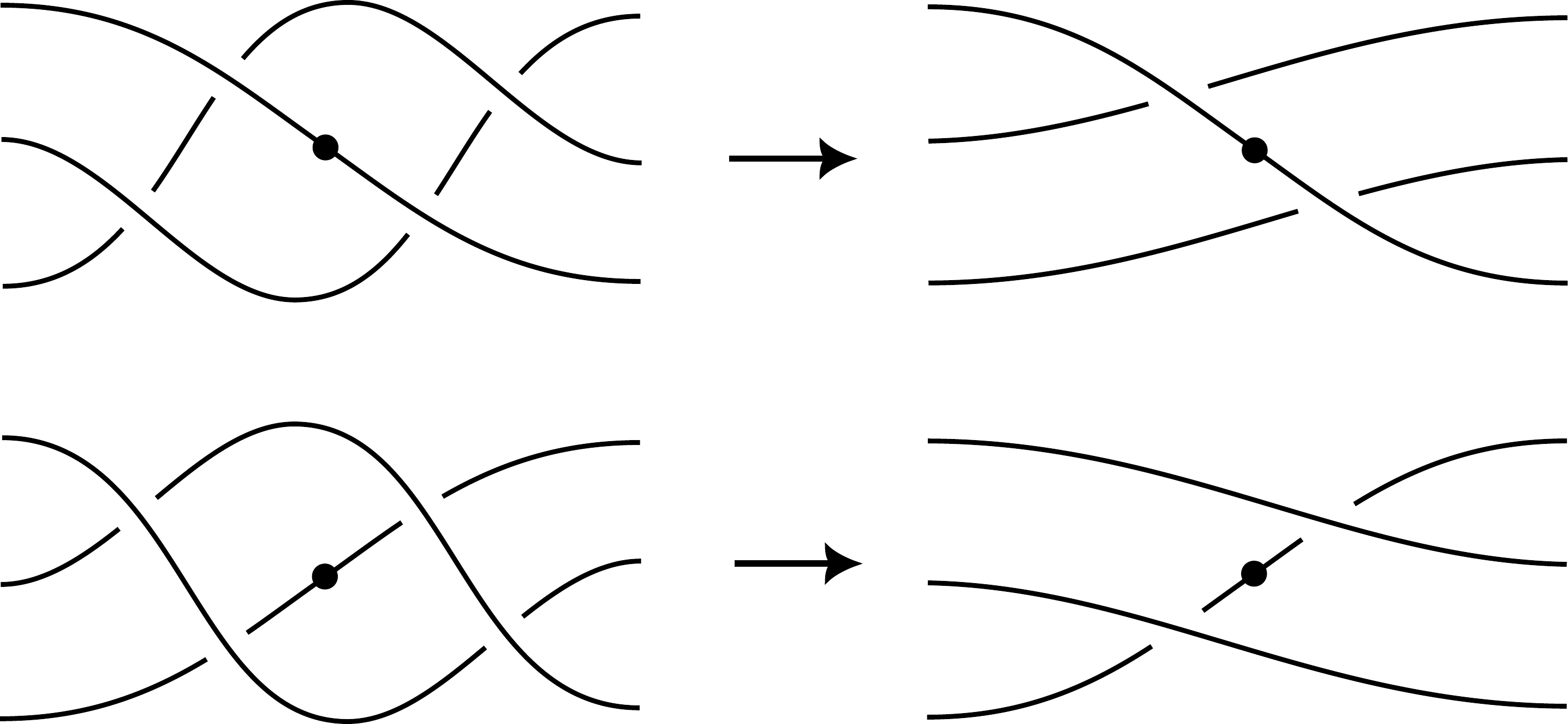}}
\caption{The middle three strands of two intravergent braids with $2n+1$ strands. On the top left is the braid corresponding to $\sigma_{n+1}\sigma_{n}\sigma_{n+1}\sigma_{n}$, which becomes $\sigma_n\sigma_{n+1}$ (top right) after performing a type B move. On the bottom left is the braid corresponding to $\sigma_{n}\sigma_{n+1}\sigma_{n}\sigma_{n+1}$, which becomes $\sigma_{n+1}\sigma_n$ (bottom right) after performing a type B move.}
\label{fig:typeBbraidmove}
\end{figure}

\item[(3.2.4)] If $i > n+1$, then consider the ending portion of $w$ consisting of $\sigma_{n+1} \sigma_{n+2} \dots \sigma_{i}$, and the subsequent letter, which by symmetry is $\sigma_{2n+1-i}$. Since $i > n+1$, we have that $\sigma_{2n+1-i}$ commutes past $\sigma_i$ symmetrically. In the original location of $w$, we now have the word $w' = \sigma_{1} \sigma_{2} \dots \sigma_{i-2}\sigma_{i-1} \sigma_{2n+1-i}$. Referring back to the proof of case (c) in Lemma \ref{lemma:sigmaplow}, we conclude that $w'$ is equivalent to the word $\sigma_{2n+2-i} \sigma_{1} \sigma_{2} \dots \sigma_{i-2}\sigma_{i-1}$. Finally, by induction on the length of $w$, we may repeat this argument until we reduce back to case (3.2.3).
\end{enumerate}
\end{enumerate}

Finally, returning to case (4), there must be two occurrences of $\sigma_1$ with index $\leq m$ or two occurrences of $\sigma_1$ with index $>m$ so that the proof is identical to case (3.1).
\end{proof}

\begin{proof}[Proof of Theorem \ref{thm:torusknots}]
First, note that torus knots are closures of positive intravergent braids. Indeed, for the torus knot $K = T(p,q)$ we can assume without loss of generality that $p$ is odd, and a positive intravergent braid with closure $K$ is given by the word
\[
\left( \prod_{i = 1}^{p-1} \sigma_i\right)^q
\]
on $p$ strands. Now by Proposition \ref{prop:weight} we have that $\widetilde{u}(K) \leq \dfrac{(p-1)q - p+1}{2} = \dfrac{(p-1)(q-1)}{2}$. Then since $u(K) \leq \widetilde{u}(K)$, and $u(K) = \dfrac{(p-1)(q-1)}{2}$ by \cite{MR1241873}, we have the desired equality $u(K) = \widetilde{u}(K)$.
\end{proof}
	

\section{Type A unknotting}

We begin by restating and proving Theorem \ref{thm:typeAunknots} and Theorem \ref{thm:typeAbounds}.

\typeAunknots*

\begin{proof}
To begin, consider an intravergent diagram for a strongly invertible knot $K$, that is a symmetric diagram where the axis of symmetry is orthogonal to the plane of the diagram. Since $K$ is strongly invertible, there is a central crossing where the knot intersects the axis of symmetry. Cutting $K$ at this central crossing, we have two arcs $\gamma_1$ and $\gamma_2$. Note that changing any symmetric pair of crossings in this diagram constitutes a type A crossing change. With a sequence of such crossing changes we can ensure that $\gamma_1$ always passes over $\gamma_2$, and furthermore that $\gamma_1$ (and hence by symmetry $\gamma_2$) is an unknotted arc. Then an isotopy reduces the diagram to only the single central crossing, which represents the unknot. An example is shown in Figure \ref{fig:typeAunknotting}.
\end{proof}

\begin{figure}
\scalebox{.25}{\includegraphics{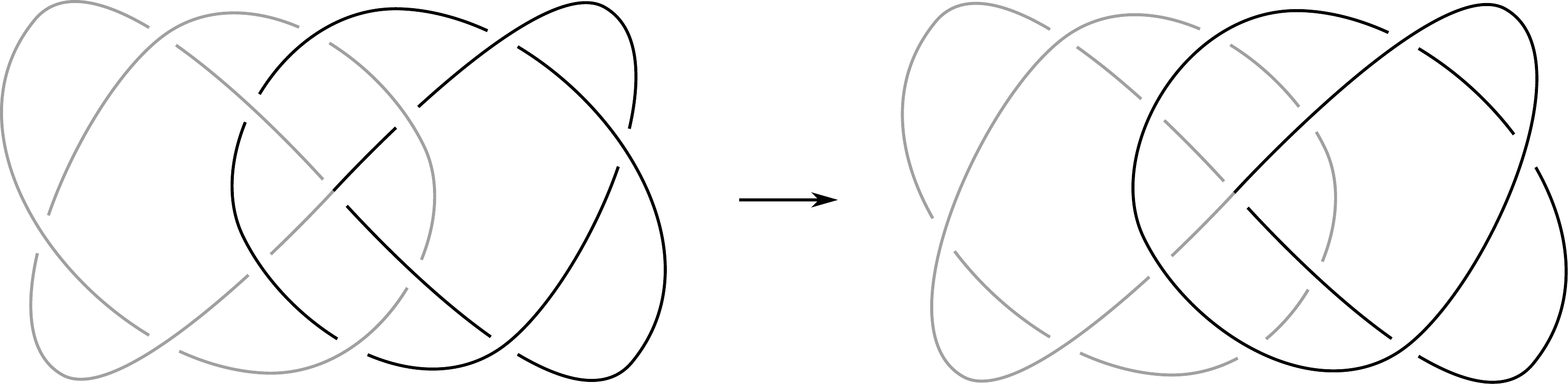}}
\caption{An example of unknotting a strongly invertible knot with only type A crossing changes, by ensuring that the black arc always passes over the gray arc, and that the black and gray arcs are unknotted.}
\label{fig:typeAunknotting}
\end{figure}

\typeAbounds*

\begin{proof}
Consider any type A crossing change pair, $\{c,\rho(c)\}$. Taking the quotient, the effect is a crossing change on $\mathfrak{q}_1(K)$ and a crossing change on $\mathfrak{q}_2(K)$. Thus any type A unknotting sequence for $K$ induces an unknotting sequence of the same length for both $\mathfrak{q}_1(K)$ and $\mathfrak{q}_2(K)$.
\end{proof}

As a consequence of this theorem, we can immediately see that the type A unknotting number can be unbounded, even for knots with a fixed unknotting number.

\begin{corollary} \label{cor:typeAquot}
For any positive integer $n$, there is a strongly invertible knot $K_n$ such that $u(K_n) = 1$, but $\widetilde{u}_A(K_n) \geq n$. In particular the difference $\widetilde{u}_A(K_n) - u(K_n)$ is unbounded.
\end{corollary}

\begin{proof}
Consider the twist knot $K_n$ shown in Figure \ref{fig:twistknots}. By changing one of the bottom off-axis crossings we can easily see that $u(K_n) = 1$. However, the quotient knot $\mathfrak{q_1}(K_n)$ is $T(2,2n+1)$, and by Kronheimer and Mrowka's proof of the Milnor conjecture \cite[Corollary 1.3]{MR1241873} we have $u(T(2,2n+1)) = n$. By Theorem \ref{thm:typeAbounds} we conclude that $\widetilde{u}_A(K_n) \geq n$.
\end{proof}

\begin{figure}
\begin{overpic}[width=200pt, grid=false]{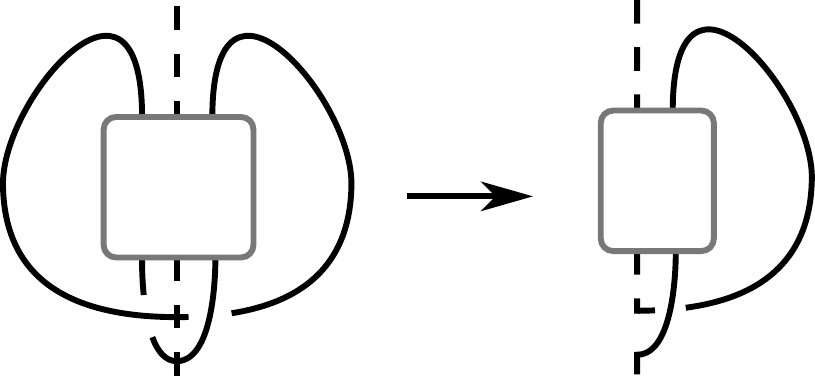}
\put (20, 22.5) {\large$n$}
\put (78.7, 22.5) {\large$n$}
\end{overpic}
\caption{A strongly invertible twist knot $K_n$ (left) and the quotient $\mathfrak{q}_1(K_n) = T(2,2n+1)$ (right) corresponding to the unbounded arc of the axis. The $n$ indicates $n$ full twists so that $K_n$ is alternating with $2n+2$ crossings.}
\label{fig:twistknots}
\end{figure}

\section{Type B unknotting}
\label{sec:B}

We begin by discussing 4-moves, which, as we shall see, are closely related to type B moves for strongly invertible knots.

\begin{definition}
A \emph{4-move} on a knot $K$ is a local tangle replacement as shown in Figure \ref{fig:4move}. The minimum number of 4-moves necessary to unknot a knot $K$ is the 4-move unknotting number $u_4(K)$. If $K$ cannot be unknotted with 4-moves, then we say that $u_4(K) = \infty$.
\end{definition}

\begin{figure}
\scalebox{.35}{\includegraphics{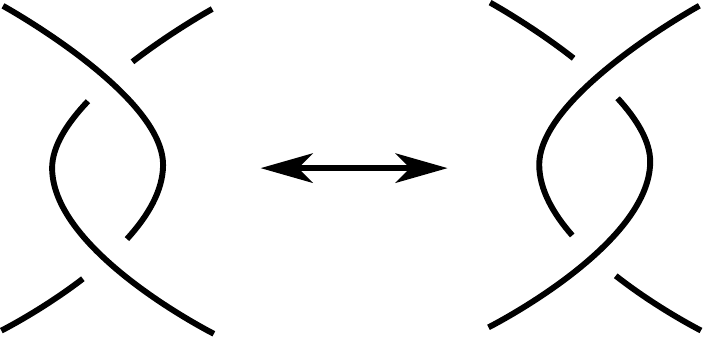}}
\caption{The 4-move as a local tangle replacement.}
\label{fig:4move}
\end{figure}

It is an old conjecture \cite[Conjecture B]{MR0881755} (which appears in the Kirby problem list \cite[1.59(3)(a)]{kirbylist}) that every knot can be unknotted with 4-moves. This conjecture has been verified through 12 crossings in \cite{MR3544442}, which also contains a history of the problem. 

\begin{lemma} \label{lemma:4moves}
Let $K$ be a strongly invertible knot with quotients $\mathfrak{q}_1(K)$ and $\mathfrak{q}_2(K)$. Applying a type B move to $K$ has the effect of applying a 4-move to one of $\mathfrak{q}_1(K)$ or $\mathfrak{q}_2(K)$, but has no effect on the other quotient knot.
\end{lemma}
\begin{proof}
Consider a type B move on $K$, and note that the crossing change must occur at a fixed point of the symmetry so that we have the situation shown in Figure \ref{fig:typeBto4move}. We then have two cases. If the type B crossing change occurs on $h_1$, then the dotted axis becomes an arc of $\mathfrak{q}_1(K)$ so that we have a 4-move on $\mathfrak{q}_1(K)$ but there is no effect on $\mathfrak{q}_2(K)$. If the type B crossing change occurs on $h_2$, then the dotted axis becomes an arc of $\mathfrak{q}_2(K)$, so that we have a 4-move on $\mathfrak{q}_2(K)$ but there is no effect on $\mathfrak{q}_1(K)$.
\end{proof}

\begin{figure}
\scalebox{.35}{\includegraphics{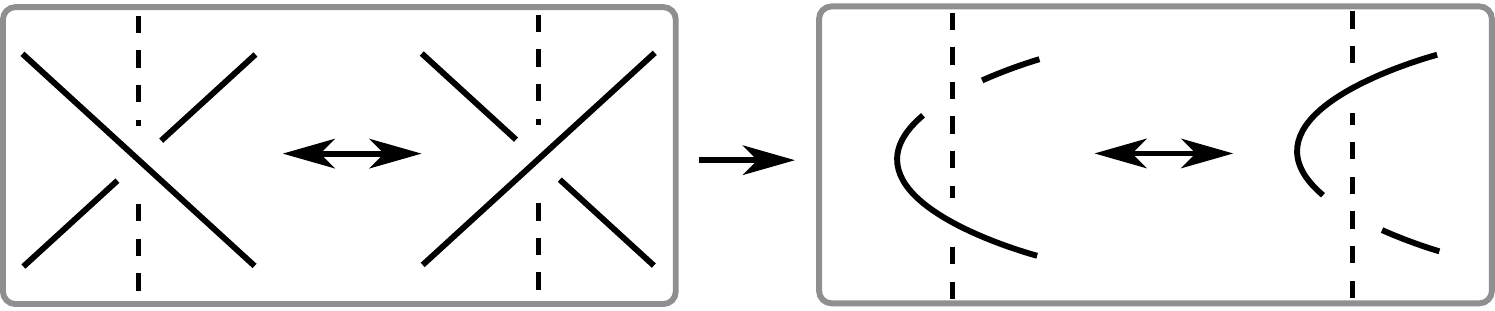}}
\caption{A type B crossing change on a strongly invertible knot $K$ (left) becomes a 4-move on $\mathfrak{q}_1(K)$ (right) when the shown section of the axis of symmetry is $h_1$.}
\label{fig:typeBto4move}
\end{figure}

\typeBbounds*
\begin{proof}
Any length $n$ unknotting sequence of type B moves on $K$ produces 4-move unknotting sequences for both $\mathfrak{q}_1(K)$ and $\mathfrak{q}_2(K)$ whose lengths sum to $n$, by Lemma \ref{lemma:4moves}.
\end{proof}

\begin{corollary} \label{cor:typeB4move}
If for every strongly invertible knot $K$ we have $\widetilde{u}_B(K) < \infty$, then every knot can be unknotted with 4-moves. 
\end{corollary}
\begin{proof}
Let $K'$ be a knot, and consider the strongly invertible knot $K = K' \# rK'$ where the strong inversion exchanges the two factors. Note that $\mathfrak{q}_1(K) = \mathfrak{q}_2(K) = K'$. Then since $\widetilde{u}_B(K) < \infty$, we have that $u_4(K') + u_4(K') < \infty$ by Theorem \ref{thm:typeBbounds}.
\end{proof}

The following theorem allows us to obstruct 2-bridge knots from having $u_4(K) = 1$. This theorem appears as Theorem 1.2 in \cite{MR4475491}. For completeness, we provide a proof below.

\begin{theorem} [Theorem 1.2 in \cite{MR4475491}] \label{thm:2bridge4moves}
Let $K$ be a 2-bridge knot corresponding to the fraction $p/q$. Then $u_4(K) = 1$ if and only if there are integers $r$ and $s$ such that
\begin{enumerate}
\item $\gcd(r,s) = 1$,
\item $4rs = \pm p \pm 1$, and
\item $\pm q^{\pm 1} \equiv 4s^2$ \textup{mod} $p$.
\end{enumerate}
\end{theorem}
\begin{proof}
For the forward direction, suppose that $u_4(K) = 1$. By an analog of the Montesinos trick, the 2-fold branched cover $\Sigma(K)$ of $S^3$ over $K$ can be obtained by $D/4$ surgery on some knot $K'$ in $S^3$. However, since $\Sigma(K) = L(p,q)$ has a cyclic fundamental group, the cyclic surgery theorem \cite{MR0881270} tells us that the exterior of $K'$ is Seifert fibered from which we can see that $\pi_1(K')$ has a non-trivial center. Then by \cite{MR0210113}, $K'$ must be a torus knot $T(r,s)$. Since $K'$ is a knot and not a link we immediately have (1). Now by Moser's theorem classifying torus knot surgeries \cite{MR0383406}, a $D/4$ surgery on a torus knot which produces a lens space must be of the form $S^3_{D/4}(T(r,s)) = L(D,4s^2)$, where $|4rs + D| = 1$. Since we in fact obtain $L(p,q)$, we have $D = \pm p$ and $4rs = \pm p \pm 1$ so that (2) is satisfied. Finally, by the classification of lens spaces, we have that $L(p,q)$ is orientation-preserving homeomorphic to one of $L(\pm p,4s^2)$ if and only if condition (3) is satisfied.

For the reverse direction, consider the the torus knot $T(r,s)$, and let $D = 4rs \pm 1$ such that $|D| = |p|$. Now $T(r,s)$ has a unique strong inversion $\rho$ which induces a symmetry which we again call $\rho$ on $S^3_{D/4}(T(r,s)) = L(D,4s^2)$. Now by \cite[Corollary 4.12]{MR0823282}, each lens space is the double branched cover of a unique link (specifically a 2-bridge link) so that the quotient of $L(D,4s^2)$ by $\rho$ must be the 2-bridge knot with fraction $p/q$. We now compare the images of the fixed-point sets in $S^3/\rho$ and $L(D,4s^2)/\rho$. In $S^3/\rho$ we have a trivial tangle consisting of two arcs in a 3-ball, which is the quotient of a neighborhood of $T(r,s)$. In $L(D,4s^2)/\rho$, the tangle becomes twisted by the $D/4$ surgery to give us 4 half twists; that is a 4-move.
\end{proof}

Before discussing applications to $\widetilde{u}_B$, we give an example of a direct application of Theorem \ref{thm:2bridge4moves} to the figure-eight knot.
\begin{example} \label{ex:4move4_1}
Consider the knot $K = 4_1$, which is the 2-bridge knot corresponding to the fraction $5/2$. We then consider the possible values of $s$ in Theorem \ref{thm:2bridge4moves}. Since $4s$ is a factor of $\pm p \pm 1 = \{-6,-4,4,6\}$, we have that $|s| = 1$. However, $\pm 2^{\pm 1} = \{2,3\}$ mod 5, so that $\pm q^{\pm 1} \not\equiv 4s^2$ mod 5. We conclude that $u_4(K) \neq 1$. On the other hand it is not hard to see that $u_4(K) \leq 2$ by directly performing two 4-moves on $K$; see Figure \ref{fig:figureeight4move}.
\end{example}

\begin{figure}
\scalebox{.3}{\includegraphics{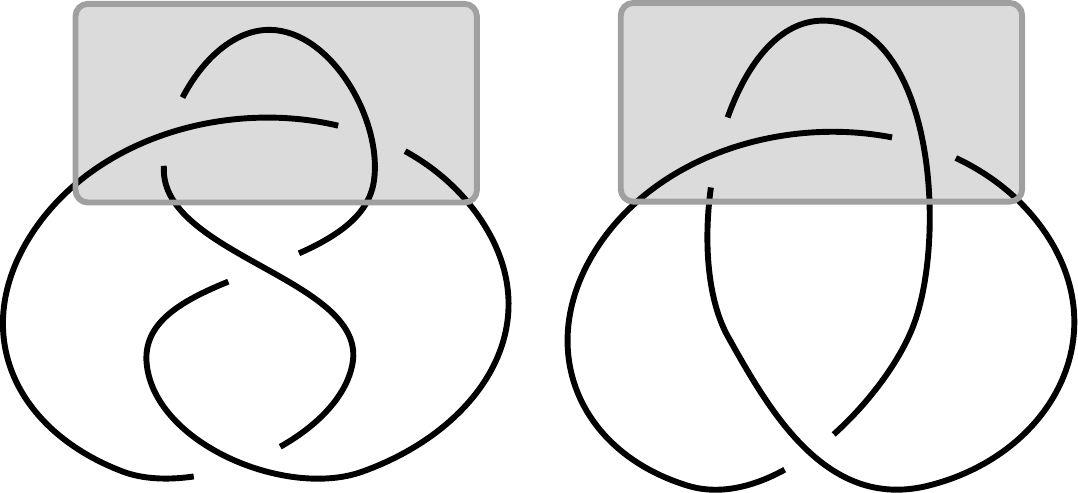}}
\caption{Performing a 4-move on the figure-eight knot (left) by changing the crossings in the clasp in the shaded box produces the trefoil. After an isotopy we get the diagram shown on the right, and performing the indicated 4-move produces the unknot.}
\label{fig:figureeight4move}
\end{figure}

By combining Theorem \ref{thm:2bridge4moves} and Theorem \ref{thm:typeBbounds} we find a knot $K$ where $\widetilde{u}_B(K) - u(K) \geq 2$.

\begin{example} \label{ex:connectsumof41s}
Consider the knot $K = 4_1\#4_1$ with the strong inversion which exchanges the two components. Then $\mathfrak{q}_1(K) = \mathfrak{q}_2(K) = 4_1$, and by Theorem \ref{thm:2bridge4moves} (see Example \ref{ex:4move4_1}) we have that $u_4(4_1) = 2$. Theorem \ref{thm:typeBbounds} then says that $\widetilde{u}_B(K) \geq 2 + 2 = 4$, in contrast with the usual unknotting number of $K$, which is 2. 
\end{example}

\section{Type C unknotting} \label{sec:C}
We start by proving Theorem \ref{thm:typeCbounds}.
\begin{definition}
The non-orientable band unknotting number $u_{nb}(K)$ of a knot $K$ is the minimum number of non-orientable band moves needed to transform $K$ into the unknot. 
\end{definition}

\typeCbounds*

\begin{proof}
We claim that a type C crossing change on $K$ descends to a non-orientable band move on the quotient theta graph which restricts to a non-orientable band move on one of $\mathfrak{q}_1(K)$ and $\mathfrak{q}_2(K)$, without affecting the other. Note that when $K$ is the unknot both $\mathfrak{q}_1(K)$ and $\mathfrak{q}_2(K)$ are  unknots, and so the theorem will follow by induction on $\widetilde{u}_C(K)$.

We now prove the claim. Consider a type C crossing change, as shown on the left in Figure \ref{fig:typeCquotient}. We have chosen a diagram so that the isotopy corresponding to the type C crossing change is compact; considering a type C move which passes through infinity in the indicated diagram would have the effect of exchanging $\mathfrak{q}_1(K)$ and $h_1$ with $\mathfrak{q}_2(K)$ and $h_2$ in what follows. The quotient theta graph admits a diagram as in the center of Figure \ref{fig:typeCquotient}. Then the indicated band move restricts to an R1 move on the induced diagram of $\mathfrak{q}_1(K)$, so that the knot type of $\mathfrak{q}_1(K)$ remains unchanged, and restricts to a non-orientable band move on $\mathfrak{q}_2(K)$. 

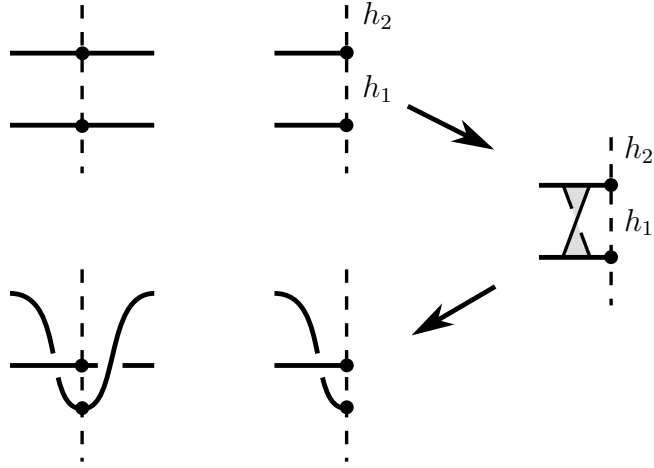
\begin{figure}
\begin{overpic}[width=230pt, grid=false]{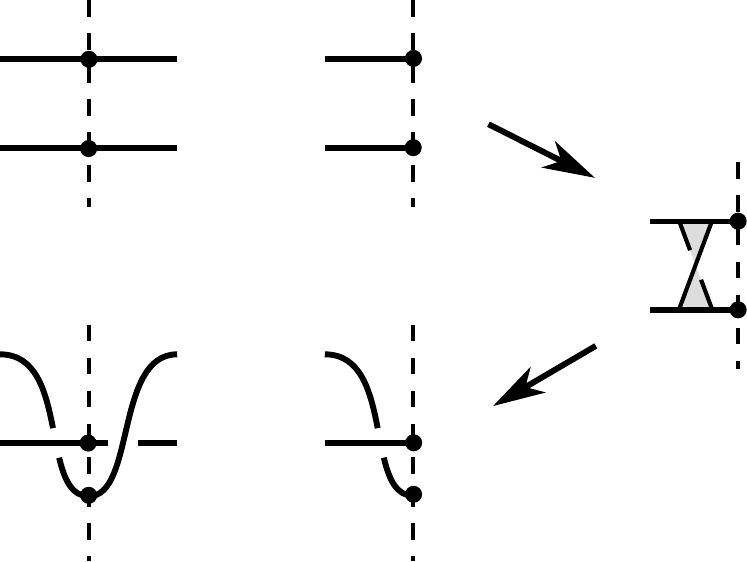}
\put (101, 38) {$h_1$}
\put (58, 60) {$h_1$}
\put (101, 50) {$h_2$}
\put (58, 72) {$h_2$}
\end{overpic}
\caption{The effect of a type C crossing change (left) on the quotient theta graph (center) of a strongly invertible knot is a non-orientable band move (right). The dashed line indicates the axis of symmetry, and the solid line indicates the knot.}
\label{fig:typeCquotient}
\end{figure}
\end{proof}

\begin{remark}
There are several straightforward lower bounds on $u_{nb}(K)$, such as the (not necessarily orientable) band-unlinking number and the non-orientable 4-genus, and lower bounds on both of these have appeared in the literature. For example, a lower bound for the band-unlinking number is given in \cite[Theorem 4]{MR1075165} in terms of the homology groups of cyclic branched coverings of $K$, and lower bounds on the non-orientable 4-genus can be found in \cite{MR3272020}, \cite{MR3694597}, and \cite{MR4640136}. Many of these lower bounds rely on knot Floer homology, but the invariants are computable in our situation. Indeed, it follow from Theorem \ref{thm:1_2typeC} that the quotients of type C unknottable knots are (1,1)-knots.

\end{remark}

\subsection{Type C unknottable knots}
In this section we classify which strongly invertible knots can be unknotted with a sequence of type C moves. Our main result is Theorem \ref{thm:1_2typeC}, which we recall here for convenience.

\typeC*

\begin{corollary} \label{cor:tunnel}
Let $K$ be a strongly invertible knot which can be equivariantly unknotted with type C moves and let $t(K)$ be the tunnel number of $K$. Then $t(K) \leq 2$.
\end{corollary}

\begin{proof}
By Theorem \ref{thm:1_2typeC}, $K$ is a $(1,2)$ knot. We will show that $(1,2)$ knots have tunnel number 2 or less. Indeed, consider the $(1,2)$ decomposition of $K$ which is the union of two handlebodies $H_1 \cup H_2$. We may connect the two boundary parallel arcs of $K \cap H_1$ with a pair of arcs $\gamma_1$ and $\gamma_2$ such that $\gamma_1 \cup \gamma_2$ is isotopic to the core of $H_1$. Removing a neighborhood of $K \cup \gamma_1 \cup \gamma_2$ from $S^3$ leaves us with a space which retracts onto $H_2$ with a pair of thickened boundary-parallel arcs removed. This is a genus 3 handlebody, so that the tunnel number of $K$ is at most 2.
\end{proof}

We can now prove Theorem \ref{thm:typeCadditive}, which we restate here for convenience.
\typeCadditive*

\begin{proof}
By \cite[Theorem 14]{MR1660345}, the tunnel number of $K_1\#K_2\#K_3$ is at least 3, so that by Corollary \ref{cor:tunnel}, $\widetilde{u}_C(K_1\#K_2\#K_3) = \infty$.
\end{proof}

\subsection{Proof of Theorem \ref{thm:1_2typeC}}
We begin with some necessary definitions and lemmas.
\begin{definition} \label{def:trivialarcs}
Let $a$ and $b$ be a pair of arcs in $D^2 \times S^1$ which are symmetric under a $\pi$ rotation around the core $\{0\} \times S^1$ of the handlebody. Considered up to isotopy relative to the four points on the boundary, we say that $a$ and $b$ are \emph{trivial} if they match exactly the arcs shown on the left in Figure \ref{fig:typeCoutcomes}, where the dotted line indicates the core of the handlebody.
\end{definition}

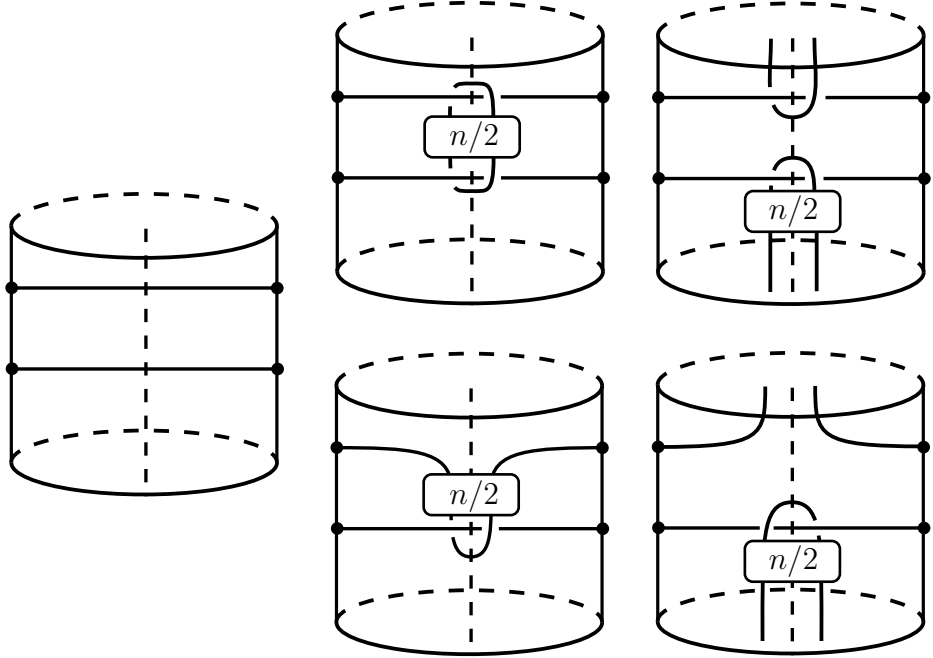
\begin{figure}
\begin{overpic}[width=350pt, grid=false]{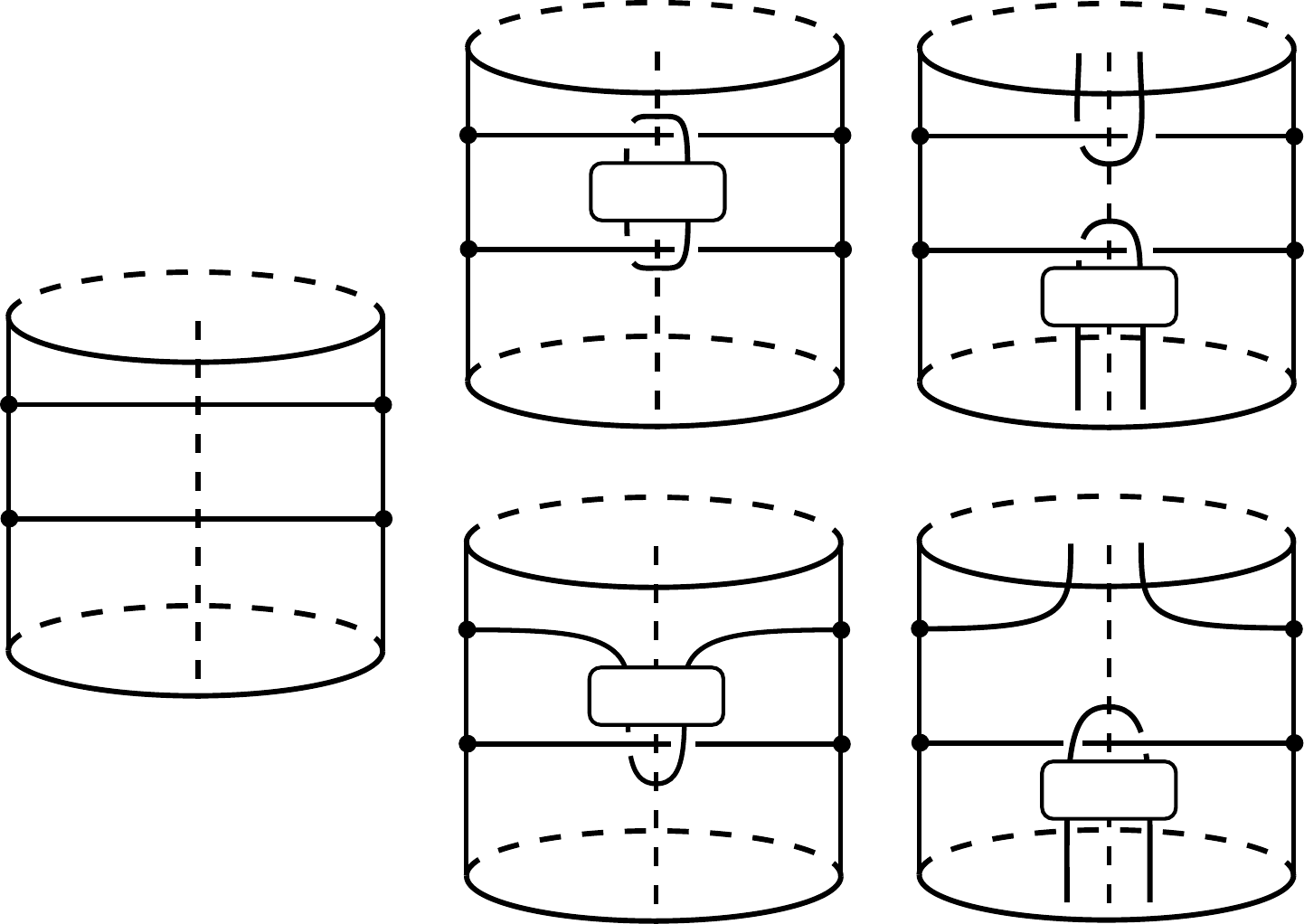}
\put (48,55.3) {$n/2$}
\put (48,16.5) {$n/2$}
\put (82.5,9.4) {$n/2$}
\put (82.5,47.3) {$n/2$}
\end{overpic}
\caption{A trivial pair of arcs in $D^2 \times S^1$ (left). The possible symmetric surgery curves corresponding to a type C crossing change are shown top center and top right, where $n/2$ indicates $n$ half-twists. The resulting tangles after the type C crossing change are shown bottom center and bottom right, respectively.}
\label{fig:typeCoutcomes}
\end{figure}
For the following lemma, note that we can perform type C crossing changes on strongly invertible tangles in $D^2 \times S^1$ (where the axis of symmetry is the core of the handlebody), using essentially the same definition as for strongly invertible knots.
\begin{lemma} \label{lem:typeCoutcomes}
Given a pair of trivial symmetric arcs in $D^2 \times S^1$, the result of a type C move, up to equivariant isotopy relative to the boundary, is one of the tangles shown in Figure \ref{fig:typeCoutcomes} in the bottom center or bottom right.
\end{lemma}

\begin{proof}
To begin, we will think of a type C move as surgery along an unknot which bounds a symmetric disk that intersects each arc once. Since this disk can intersect each arc only once and is symmetric, the disk must intersect the arcs at their fixed points. Now any order 2 symmetry of a disk has a contractible fixed set by classical results of Smith (see for example \cite[Corollary 1.3.8]{MR1236839}). Since we have at least two fixed points (one on each arc), the fixed set of the disk must be an arc. There are exactly two fixed arcs connecting the two known fixed points, so the disk retracts to an interval bundle over an arc contained in the axis of symmetry. The two resulting possibilities are shown in the top center or top right of Figure \ref{fig:typeCoutcomes}. Finally, performing $+1$-surgery along these curves produces the arcs shown in the bottom center and bottom right of Figure \ref{fig:typeCoutcomes}.
\end{proof}

Let $T^2_i$ be the torus with $i$ punctures. Let $\MCG(T^2_2)$ be the mapping class group of the twice-punctured torus $T^2_2$. Let $\MCG((S^1 \times D^2)_2)$ be the mapping class group of $S^1 \times D^2$ which preserves setwise a pair of marked points on the boundary. Let $\SMCG((S^1 \times D^2)_4)$ be the symmetric mapping class group of $S^1 \times D^2$ consisting of diffeomorphisms which respect the symmetry $\rho$ given by a $\pi$ rotation on the $D^2$ component and preserve a $\rho$-invariant set of four marked points. Finally, let $\LMCG((S^1 \times D^2)_2)$ be the image of $\SMCG((S^1 \times D^2)_4)$ in $\MCG((S^1 \times D^2)_2)$ under the map induced by the quotient. 

The following lemma can be thought of as a Birman-Hilden-type theorem (see e.g. \cite{MR4275077}) for the two-fold branched covering of a genus one handlebody with two marked points over its core.
\begin{lemma} \label{lem:mcgiso}
There is a short exact sequence of groups 
\[
1 \to \mathbb{Z}/2\mathbb{Z} \overset{i}{\to} \SMCG((S^1 \times D^2)_4) \overset{p}{\to} \LMCG((S^1 \times D^2)_2) \to 1,
\]
where $i$ is the inclusion of the subgroup generated by the mapping class of $\rho$, and $p$ is the map induced from the quotient map of the symmetry $\rho$.
\end{lemma}

\begin{proof}
First, observe that $p \circ i$ is the trivial map, since $\rho$ projects to the identity map on $(S^1 \times D^2)_2$. It remains to check that the kernel of $p$ is the subgroup generated by $[\rho]$ in $\SMCG((S^1 \times D^2)_4)$. Let $f \in \SMCG((S^1 \times D^2)_4)$ such that $p(f)$ is trivial in $\LMCG((S^1 \times D^2)_2)$. By composing with an equivariant diffeomorphism near the core if necessary, we may assume that $f$ fixes the core $S^1 \times \{0\}$ pointwise. We show below that there is an isotopy $H$ from $p(f)$ to the identity map $Id$ such that at each time $t$, $H_t$ fixes the core setwise. Since $H$ fixes the core, we can lift $H$ to an equivariant isotopy on $(S^1 \times D^2)_4$ from $f$ to an element in $p^{-1}(Id)$. Hence $f$ is in the subgroup generated by $[\rho]$.

To construct $H$, first note that since $p(f)$ is isotopic to the identity, we may isotope $p(f)$ near the boundary so that the boundary is pointwise fixed. This isotopy also fixes a neighborhood of the core pointwise. Now choose an annulus $A$ embedded in $S^1 \times D^2$ such that the one boundary component of $A$ is the core, and the other boundary component lies on $\partial (S^1 \times D^2)$. Now consider $A$ and $p(f)(A)$. By a standard argument applied to the intersection between $A$ and $p(f)(A)$, we can isotope across 3-balls and handlebodies bounded by $A \cup p(f)(A)$ to modify $p(f)$ such that $p(f)(A) = A$, and these isotopies fix pointwise the core and the boundary. Next, we isotope $p(f)$ (fixing $\partial(S^1 \times D^2)$ pointwise, but applying full twists to the core as necessary) so that $p(f)|_A$ is the identity map. The problem is now reduced to finding an isotopy relative to $(S^1 \times S^1) \cup A$ between $p(f)$ and $Id$. To do this, first isotope $p(f)$ to be identity on a neighborhood $\nu(A \cup (S^1 \times S^1))$. Then recall that the mapping class group of $S^1 \times D^2$, which is diffeomorphic to $(S^1 \times D^2) - (\nu(A \cup S^1 \times S^1))$, is a subgroup of the mapping class group of its boundary. Hence there is an isotopy between $p(f)|_{(S^1 \times D^2) - (\nu(A \cup S^1 \times S^1))}$ and $Id$, which fixes at each time its boundary. Gluing this to the identity isotopy on $\nu(S^1 \times S^1 \cup A)$, we have an isotopy $H$ from $p(f)$ to $Id$ which fixes pointwise at each time the boundary preserves the core setwise.

\end{proof}

\begin{figure}
\begin{overpic}[width=300pt, grid=false]{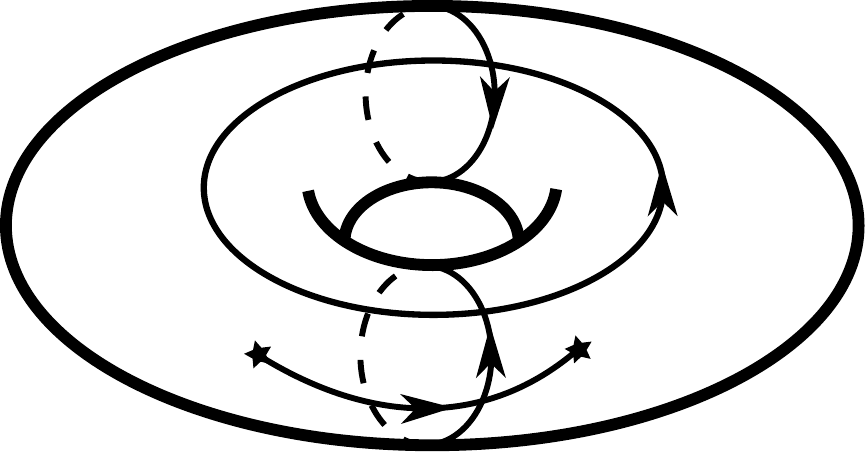}
\put (58,36) {$\beta$}
\put (79,28) {$\alpha$}
\put (58,11) {$\gamma$}
\put (48,6.5) {$\delta$}
\end{overpic}
\caption{The torus with two marked points $(S^1 \times S^1)_2$ with the indicated curves $\alpha, \beta, \gamma$, and $\delta$. When thought of as the boundary of $S^1 \times D^2$, the curves $\beta$ and $\gamma$ bound disks.}
\label{fig:MCGgens}
\end{figure}
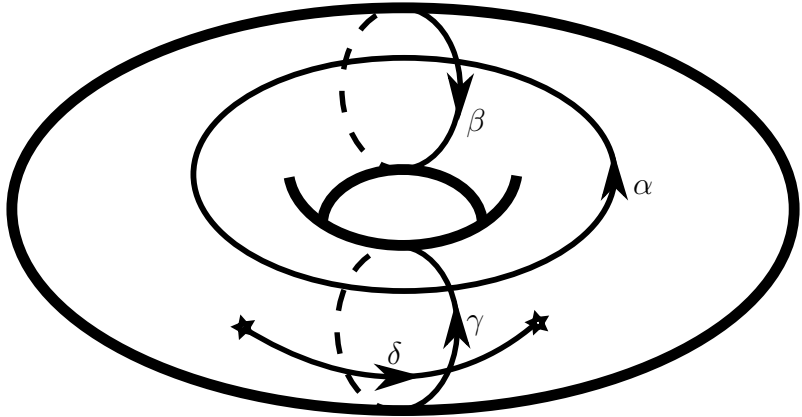

Before proceeding, we define some elements of $\MCG((S^1 \times S^1)_2)$ which we will use to write down a generating set for $\LMCG((S^1 \times D^2)_2)$. We will need the following specific diffeomorphisms, which we will also use to represent the corresponding isotopy classes of diffeomorphisms. The definitions refer to the curves indicated in Figure \ref{fig:MCGgens}.
\begin{enumerate}
\item The diffeomorphisms $t_\alpha$, $t_\beta$ and $t_\gamma$ are given by Dehn twists around the curves $\alpha$, $\beta$ and $\gamma$ respectively.
\item The diffeomorphism $\tau$ is given by $t_\alpha t_\beta t_\alpha t_\gamma t_\alpha t_\beta$, which is the hyperelliptic involution which fixes the two marked points, preserves $\alpha$, and swaps $\beta$ and $\gamma$.
\item The diffeomorphism $\sigma$ is the identity map outside of a neighborhood of $\delta$, and swaps the marked points by a clockwise $180^{\circ}$ rotation within a neighborhood of $\delta$.
\item The diffeomorphism $m$ is given by $t_\beta t_\gamma^{-1}$, which pulls one marked point around a loop parallel to $\beta$.
\item The diffeomorphism $\ell$ is given by $t_\gamma t_\beta^{-1} t_\alpha t_\beta t_\gamma^{-1} t_\alpha^{-1}$, which pulls one marked point around a loop parallel to $\alpha$.
\end{enumerate}
It will also be useful to remember that $\MCG((S^1 \times D^2)_2)$ is exactly the subgroup of $\MCG((S^1 \times S^1)_2)$ consisting of elements which map a meridian to a meridian. Hence we can slightly abuse notation and use (compositions of) the above diffeomorphisms to refer to elements of $\MCG((S^1 \times D^2)_2)$, whenever a meridian is taken to a meridian.

\begin{lemma} \label{lemma:PMCGgens}
The pure mapping class group $\PMCG((S^1 \times D^2)_2)$ is generated by $t_{\beta}$, $\tau$, $m$, and $\ell$. Moreover, every element of $\PMCG((S^1 \times D^2)_2)$ is of the form $w \cdot t_\beta^i \cdot \tau^j$, where $w$ is a word in $m$ and $l$, $i \in \mathbb{Z}$, and $j \in \{0,1\}$.
\end{lemma}
\begin{proof}
Consider the natural map $\Psi \colon \PMCG((S^1 \times S^1)_2) \to \MCG(S^1 \times S^1)$ given by forgetting the marked points. Then $\PMCG((S^1 \times D^2)_2) = \Psi^{-1}(\langle \Psi(t_\beta), \Psi(\tau)\rangle)$. To see this we will check that $\Psi(t_\beta)$ and $\Psi(\tau)$ generate the subgroup of $\MCG(S^1 \times S^1)$ consisting of elements which send meridians to meridians. Recalling that 
\[
\MCG(S^1 \times S^1) \cong \textup{SL}_2(\mathbb{Z})\colon \mathbb{Z}\langle\beta\rangle \oplus \mathbb{Z}\langle\alpha\rangle \to \mathbb{Z}\langle\beta\rangle \oplus \mathbb{Z}\langle\alpha\rangle,
\]
we have 
\[
\Psi(t_\beta) = \begin{bmatrix} 1 &1 \\ 0 & 1\end{bmatrix}, \mbox{ and } \Psi(\tau) = \begin{bmatrix} -1 &0 \\ 0 & -1\end{bmatrix},
\]
which generate all matrices preserving $\langle \beta \rangle$; that is matrices of the form
\[
\begin{bmatrix} \pm1 & n \\ 0 & \pm1\end{bmatrix}.
\]
Now we observe that $\Psi^{-1}(\langle \Psi(t_\beta), \Psi(\tau)\rangle)$ is generated by $t_\beta$, $\tau$, and ker$(\Psi)$. However ker$(\Psi) = \langle m, \ell\rangle$; see for example \cite[Proposition 1]{MR2055680}. Hence $\PMCG((S^1 \times D^2)_2) = \langle t_\beta, \tau, m,\ell \rangle$. Finally, since every element of $\PMCG((S^1 \times D^2)_2)$ is a preimage of an element in $\MCG(S^1 \times S^1)$ of the form $\Psi(t_\beta^i \tau^j)$, each element can be written as $w \cdot t_\beta^i \tau^j$ for some element $w$ in ker$(\Psi) = \langle m, \ell \rangle$.
\end{proof}
Before the next lemma, we recall a special case of the liftability criteria from \cite[Proposition 4.4]{MR4071371}. We will use the notation from Figure \ref{fig:MCGgens}, as well as the following.

Let $\LPMCG(X_n)$ denote the liftable pure mapping class group of $X$ with $n$ marked points. We will be interested in $X_n = (S^1 \times S^1)_2$ and $X_n = (S^1 \times D^2)_2$.

Let $p$ be the projection map corresponding to the 2-fold cover from the 4-punctured torus $(S^1 \times S^1)_4$ to the 2-punctured torus $(S^1 \times S^1)_2$ as induced by the symmetry $\rho$ above. Let 
\[
q \colon H_1((S^1 \times S^1)_2;\mathbb{Z}) \to H_1((S^1 \times S^1)_2;\mathbb{Z})/p_*H_1((S^1 \times S^1)_4;\mathbb{Z}) \cong \mathbb{Z}/2\mathbb{Z},
\]
be the indicated map on homology. Note that $q(\beta) = q(\gamma) = 1$ and $q(\alpha) = 0$. Given a diffeomorphism $f \colon (S^1 \times S^1)_2 \to (S^1 \times S^1)_2$, let $a_f \in H_1((S^1 \times S^1)_2;\mathbb{Z})$ be the element $a_f = f(\delta) - \delta$ in the first homology group of $(S^1 \times S^1)_2$ relative to the punctures. Here $\delta$ is the homology class represented by the arc $\delta$ in Figure \ref{fig:MCGgens}.

\begin{lemma}[Corollary of Proposition 4.4 of \cite{MR4071371}] \label{lemma:liftabilitycondition}
The mapping class of $f$ is in $\LPMCG((S^1 \times S^1)_2)$ if and only if $qf_*(\alpha) = 0, qf_*(\beta) = qf_*(\gamma) = 1$, and $q(a_f) = 0$.
\end{lemma}

\begin{lemma}
The group $\LPMCG((S^1 \times D^2)_2)$ is generated by $t_\beta^2, m^2, m\ell m^{-1}, \ell,$ and $\tau$.
\end{lemma}
\begin{proof}
We first observe that $\LPMCG((S^1 \times D^2)_2) = \LMCG((S^1 \times S^1)_2) \cap \PMCG((S^1 \times D^2)_2)$, by viewing $\MCG((S^1 \times D^2)_2)$ as a subgroup of $\MCG((S^1 \times S^1)_2)$. Then for any element $x$ in $\PMCG((S^1 \times D^2)_2)$, we can verify that $x \in \LPMCG((S^1 \times D^2)_2)$ using the conditions in Lemma \ref{lemma:liftabilitycondition}. By Lemma \ref{lemma:PMCGgens}, an arbitrary element of $\LPMCG((S^1 \times D^2)_2)$ has the form $f = wt_\beta^i \tau^j$ where $w$ is a word in $m$ and $\ell$, $i \in \mathbb{Z}$, and $j \in \{0,1\}$. Observe that $q(w(x)) = q(x)$ for all $x \in H_1((S^1 \times S^1)_2;\mathbb{Z})$, since $q$ factors through $H_1(S^1 \times S^1;\mathbb{Z})$, on which $w$ acts trivially.
We then compute
\begin{align*}
q w t_\beta^i \tau^j (\alpha) &= q w t_\beta^i ((-1)^j \alpha) = q w t_\beta^i (\alpha) = q w (\alpha + i\beta) = q(\alpha + i\beta) = i, \\
q w t_\beta^i \tau^j (\beta) &= q w t_\beta^i ((-1)^j \beta) = q w t_\beta^i (\beta) = q w (\beta) = q(\beta) = 1, \\
q w t_\beta^i \tau^j (\gamma) &= q w t_\beta^i ((-1)^j \gamma) = q w t_\beta^i (\gamma) = q w (\gamma) = q(\gamma) = 1, \\
\end{align*}
so that $i$ must be even for $w t_\beta^i \tau^j$ to be liftable. For last condition in Lemma \ref{lemma:liftabilitycondition}, we will need to keep track of the algebraic number of $m$'s and $\ell$'s appearing in $w$, call these values $|w|_m$ and $|w|_\ell$ respectively. Also, observe that $\tau$ is liftable by checking with Lemma \ref{lemma:liftabilitycondition}. Indeed $\tau(\delta) = \delta - \alpha$ so that $a_\tau = -\alpha$ and $q(-\alpha) = 0$. Then we have that $w t_\beta^i \tau$ is liftable if and only if $w t_\beta^i$ is liftable. Additionally, note that for any $x \in H_1((S^1 \times S^1)_2;\mathbb{Z})$, we have $m(x + \delta) = x + \gamma + \delta$ and $\ell(x + \delta) = x + \alpha + \delta$, noting that any loop around a puncture is trivial in the relative homology group. We can then check
\[
w t_\beta^i(\delta) = w(\delta) = \delta + |w|_m\gamma + |w|_\ell \alpha,
\] 
and so $q(a_{w t_\beta^i}) = |w|_m$. We conclude that $w t_\beta^i\tau^j$ is liftable if and only if $i$ and $|w|_m$ are both even. Words of this form are generated by $t_\beta^2, m^2, m \ell m^{-1}, \ell$ and $\tau$, as desired.
\end{proof}

We now choose preimages of $t_\beta^2, m^2, m\ell m^{-1}, \ell$, and $\tau$ in $\SMCG((S^1 \times D^2)_4)$; see Lemma \ref{lem:mcgiso}. We choose respectively the preimages $t_{\widetilde{\beta}}, \widetilde{m}, \widetilde{\ell}_+,$ and  $\widetilde{\ell}_-,$ as described in Figure \ref{fig:SMCGgens}, and $\widetilde{\tau}$, which is a hyperelliptic involution fixing the four marked points. We will also need the diffeomorphisms $\widetilde{\sigma}$, which is the lift of $\sigma$ given by swapping the four marked points in pairs, and $\rho$, which is the deck transformation involution on $(S^1 \times D^2)_4$.

\begin{figure}
\begin{overpic}[width=420pt, grid=false]{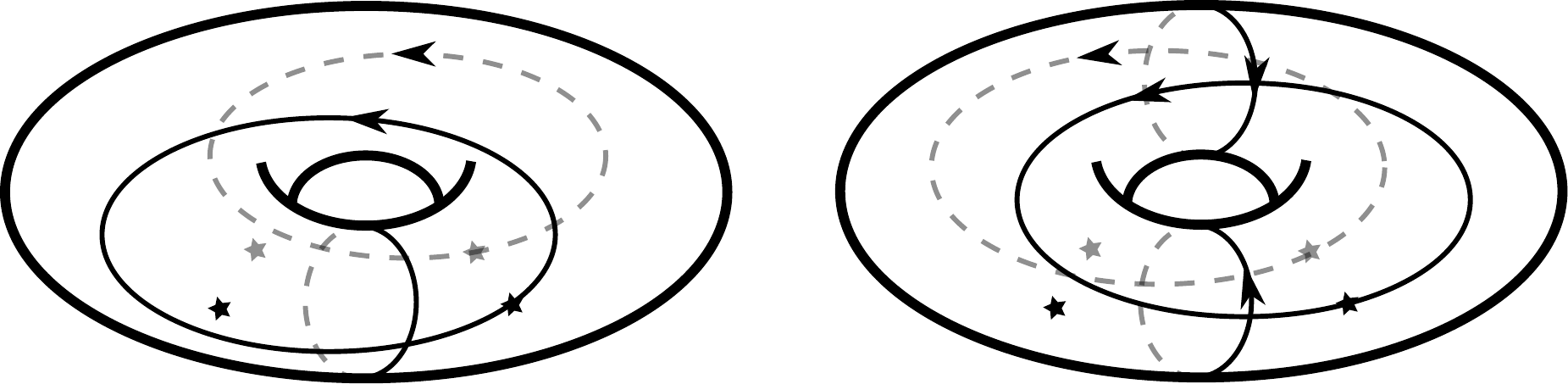}
\put (76.5,7.1) {$\widetilde{\gamma}$}
\put (80, 15.8) {$\widetilde{\beta}$}
\end{overpic}
\caption{Some generators of $\SMCG((S^1 \times D^2)_4)$. The generator $t_{\widetilde{\beta}}$ is given by a Dehn twist around the meridian $\widetilde{\beta}$, the generator $\widetilde{m}$ is given by the composition $t_{\widetilde{\beta}}t_{\widetilde{\gamma}}^{-1}$, the generator $\widetilde{\ell}_+$ (left) is given by dragging the right two marked points around the indicated longitudes, and the generator $\widetilde{\ell}_-$ (right) is given by dragging the right two marked points around the indicated longitudes.}
\label{fig:SMCGgens}
\end{figure}

\begin{proposition} \label{prop:smcggens}
The group $\SMCG((S^1 \times D^2)_4)$ is generated by $t_{\widetilde{\beta}}, \widetilde{m}, \widetilde{\ell}_+, \widetilde{\ell}_-, \widetilde{\tau},\widetilde{\sigma},$ and $\rho$.
\end{proposition}

\begin{proof}
Note that there is an exact sequence
\[
1 \to \LPMCG((S^1 \times D^2)_2) \to \LMCG((S^1 \times D^2)_2) \overset{f}{\to} \mathbb{Z}/2\mathbb{Z} \to 1,
\]
such that $f(\sigma)$ is a generator for the $\mathbb{Z}/2\mathbb{Z}$. Hence by Lemma \ref{lemma:PMCGgens}, $\LPMCG((S^1 \times D^2)_2)$ is generated by $t_\beta^2, m^2 m\ell m^{-1}, \ell, \tau,$ and $\sigma$. Now by Lemma \ref{lem:mcgiso}, we have that $\SMCG((S^1 \times D^2)_4)$ is generated by lifts of the generators of $\LMCG((S^1 \times D^2)_2)$ together with $\rho$. In particular, $t_{\widetilde{\beta}}, \widetilde{m}, \widetilde{\ell}_+, \widetilde{\ell}_-, \widetilde{\tau},\widetilde{\sigma},$ and $\rho$ generate $\SMCG((S^1 \times D^2)_4)$.
\end{proof}

\begin{lemma} \label{lem:12typeC}
Let $a$ and $b$ be the trivial pair of $\rho$-invariant arcs in $((S^1 \times D^2)_4,\rho)$ as shown on the left in Figure \ref{fig:typeCoutcomes}. Let $w \in \SMCG((S^1 \times D^2)_4,\rho)$ so that we may consider new boundary parallel arcs $w(a \cup b)$. Then $w(a \cup b)$ is related to $a \cup b$ by a sequence of type C moves.
\end{lemma} 

\begin{proof}
We will write $w$ as a word in the generators $t_{\widetilde{\beta}}, \widetilde{m}, \widetilde{\ell}_+, \widetilde{\ell}_-, \widetilde{\tau},\widetilde{\sigma},$ and $\rho$ and proceed by induction on the length of the word $w$. The base case is trivial. 

For the inductive step, suppose we have a word $v$ such that $v(a \cup b)$ is a tangle obtained by a sequence $C_n \circ \dots \circ C_1$ of type C moves applied to $a \cup b$. By Proposition \ref{prop:smcggens}, we consider $w = g \circ v$ for $g \in \{t_{\widetilde{\beta}}, \widetilde{m}, \widetilde{\ell}_+, \widetilde{\ell}_-, \widetilde{\tau},\widetilde{\sigma},\rho\}$. Then $w(a \cup b)$ is obtained from $g(a \cup b)$ by the sequence of type C moves $g(C_n) \circ \dots \circ g(C_1)$. Noting that $g(C_i)$ is a type C move, since $g$ is a diffeomorphism, it remains to check that $g(a \cup b)$ is obtained from $a \cup b$ by type C moves. In the case where $g \in \{t_{\widetilde{\beta}}, \widetilde{m}, \widetilde{\tau},\rho\}$, observe that $g(a \cup b) = a \cup b$, so that $g(a \cup b)$ is obtained by zero type C moves. In the case where $g \in \{\ell_-,\sigma\}$, we have that $\sigma(a \cup b)$ is the tangle shown in the bottom center of Figure \ref{fig:typeCoutcomes} (with zero twists), and hence is obtained from $a \cup b$ by a single type C move. In the case where $g = \ell_+$, we have that $\sigma(a \cup b)$ is the tangle shown in the bottom center of Figure \ref{fig:typeCoutcomes} (with $n = -1$ half twists), and hence is also obtained from $a \cup b$ by a single type C move. Therefore $w(a \cup b)$ is obtained from $a \cup b$ by a sequence of type C moves.
\end{proof}

\begin{proof}[Proof of Theorem \ref{thm:1_2typeC}]
For the forward direction, we proceed by induction on the number of type C moves needed to unknot $K$. For the unknot, the result is clear.

For the inductive step, suppose that we have a strongly invertible knot $K'$ with a $(1,2)$ decomposition with the axis of symmetry as the core of one of the handlebodies $H$. Without loss of generality, we may assume that any type C move on $K'$ is supported in $H$. By an appropriate diffeomorphism we may further assume that the pair of arcs of $K'$ contained in $H$ are trivial in the sense of Definition \ref{def:trivialarcs}. Then by Lemma \ref{lem:typeCoutcomes}, the result of a type C move on $H$ is as shown in the bottom center or bottom right of Figure \ref{fig:typeCoutcomes}. Since these tangles are all boundary parallel, any knot $K''$ obtained by a type C move on $K'$ still has a decomposition into a pair of boundary parallel arcs in $H$, and a pair of boundary parallel arcs in the complement of $H$. In other words, $K''$ is also a strongly invertible $(1,2)$ knot in which the core of one handlebody is the axis of symmetry. 

The reverse direction is an immediate corollary of Lemma \ref{lem:12typeC}.
\end{proof}

\section{The total equivariant unknotting number}





The main purpose of this section is to prove Theorem \ref{thm:nonadditive}, which we restate now.
\nonadditive*
\begin{proof}
We will show that for $K_1 = K_2 = T_3$, the $3$-twist knot, the equivariant connect sum $K_1 \#K_2$ shown on the left in Figure \ref{fig:twistknotsumandquotients} with $n = 3$ has equivariant unknotting number $\widetilde{u}(K_1\#K_2) \geq 3$. Since $T_3$ with the indicated strong inversion (as seen in Figure \ref{fig:twistknots}) can be unknotted with a single type C move, this will prove the theorem. 

To begin, note that $q_2(K_1 \# K_2)$, as shown on the right in Figure \ref{fig:twistknotsumandquotients}, is $T(2,7) \#T(2,7)$, which has signature $6 + 6 = 12$, so that the unknotting number of $q_2(K_1 \# K_2)$ is at least 6. Now note that a type A move on $K_1 \# K_2$ produces a crossing change on $q_2(K_1 \# K_2)$, and a type B move on $K_1 \# K_2$ produces a 4-move on $q_2(K_1 \# K_2)$ (see Theorem \ref{thm:typeBbounds}). Hence if $K_1 \# K_2$ can be unknotted with two equivariant crossing changes in the form of a type A move or two type B moves, then $q_2(K_1 \# K_2)$ can be unknotted with at most 4 crossing changes. Since $\sigma(q_2(K_1 \# K_2)) = 12$, the unknotting number of $q_2(K_1 \# K_2)$ is at least 6, and hence $\widetilde{u}_A(K_1\#K_2) \geq 6$ by Theorem \ref{thm:typeAbounds} and $\widetilde{u}_B(K_1\#K_2) \geq 2$ by Theorem \ref{thm:typeBbounds}. We conclude that $K_1 \# K_2$ cannot be unknotted with a type A or two type B moves. Hence an unknotting sequence in fewer than three equivariant crossing changes can only consist of two type C moves, or a type B and a type C move.

We will now show that any knot $J$ obtained from $K_1 \# K_2$ by a type C move cannot be unknotted with a single type B or C move, from which we will conclude that $\widetilde{u}(K_1 \# K_2) \geq 3$. There are two infinite families of knots which can be obtained from $K_1 \# K_2$ by a type C move: the knots $J_m^+$ as shown in Figure \ref{fig:twistknottypeC}, and the knots $J_m^-$ obtained from a type C move along the unbounded half-axis. Note that the quotients $q_2(J_m^-)$ are all isotopic to the quotient $q_2(K_1\#K_2) = T(2,7) \# T(2,7)$. As above, this implies that $J_m^-$ cannot be unknotted with a single type B move. Furthermore, since $T(2,7)\#T(2,7)$ is not the quotient of a twist knot, we have that $J_m^-$ cannot be unknotted with a single type C move. 

\begin{figure} 
\begin{overpic}[width=300pt, grid=false]{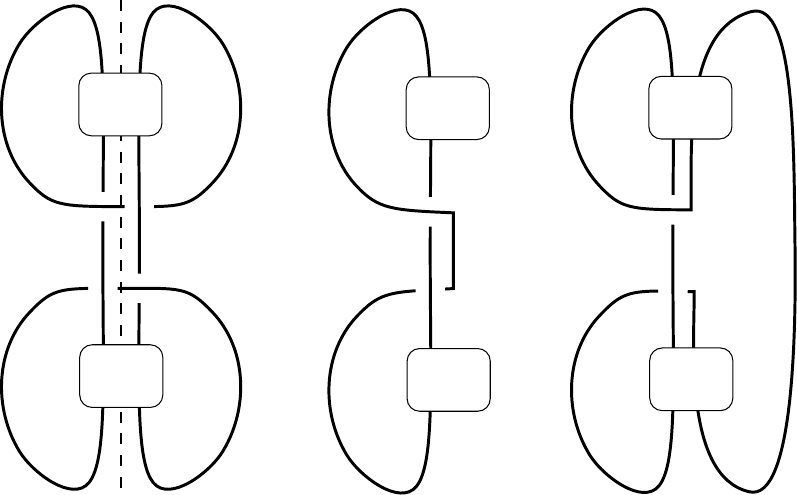}
\put (14, 14) {$n$}
\put (14, 48) {$n$}

\put (55, 13.5) {$n$}
\put (55, 47.5) {$n$}

\put (85.5, 13.5) {$n$}
\put (85.5, 47.7) {$n$}
\end{overpic}
\caption{When $n = 3$, the knot $K_1\#K_2$ (left), and the two quotient knots $q_1(K_1\#K_2)$ (center) and $q_2(K_1\#K_2)$ (right).}
\label{fig:twistknotsumandquotients}
\end{figure}

\begin{figure} 
\begin{overpic}[width=300pt, grid=false]{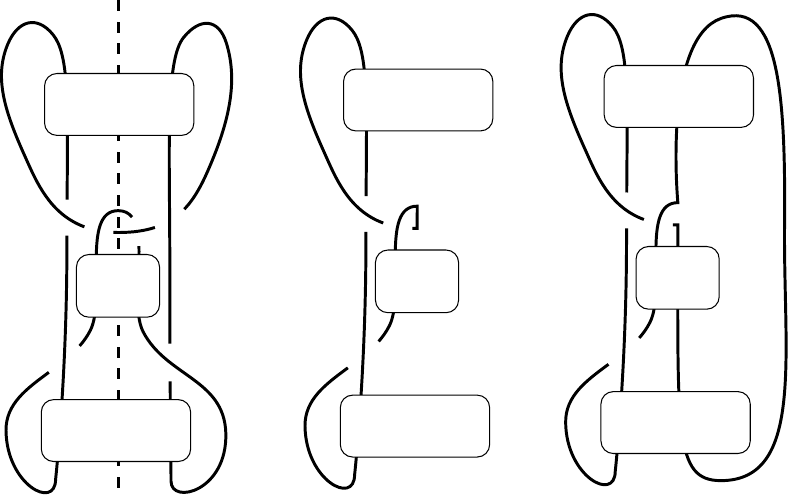}
    \put (14, 7) {$n$}
    \put (13.5, 25.5) {$m$}
    \put (14, 48.5) {$n$}

    \put (51.5, 7.5) {$n$}
    \put (51.5, 26) {$m$}
    \put (52, 49) {$n$}

    \put (85, 8) {$n$}
    \put (84.5, 26.3) {$m$}
    \put (85, 49.3) {$n$}
  \end{overpic}
\caption{The knot $J_m^+$ obtained from $K_1\#K_2$ by performing a type C move with $m$ twists as indicated (left), and the two quotients $q_1(J_m^+)$ (center) and $q_2(J_m^+)$ (right).}
\label{fig:twistknottypeC}
\end{figure}
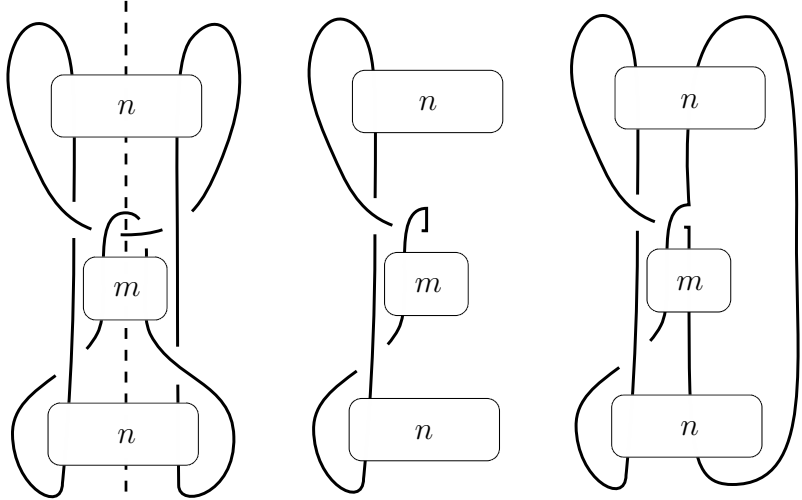

It remains to check that $J_m^+$ cannot be unknotted with a single type B, or C move. We first observe that $q_2(J_m^+)$, shown on the right in Figure \ref{fig:twistknottypeC} is the 2-bridge knot corresponding to the continued fraction 
\[
[6,-1,2m+1,-1,6] = \dfrac{7(14m+19)}{2(7m+10)}.
\]
Since this fraction is not equivalent to $\dfrac{2k+1}{1}$ for any $k$, we conclude that $q_2(J_m^+)$ is never a torus knot, and hence that $J_m^+$ is never a twist knot. Since only twist knots can be unknotted with a single type C move, $\widetilde{u}_C(J_m^+) > 1$. We next compute the signature $\sigma(q_2(J_m^+))$. The Goeritz matrix is 
\[
G = \begin{bmatrix}
7 & -1 & 0 \\
-1 & 2m + 3 & -1 \\
0 & -1 & 7
\end{bmatrix},
\]
with correction term $2m + 3$, so that $\sigma(q_2(J_m^+)) = \sigma(G) - (2m + 3)$. Since $G$ is a $3\times 3$ matrix, we have that $-2m - 6 \leq \sigma(q_2(J_m^+)) \leq -2m$. Whenever $|\sigma(q_2(J_m^+))| \geq 6$, we have that $u(q_2(J_m^+)) > 2$ so that $q_2(J_m^+)$ cannot be unknotted with a single crossing change, or with a single 4-move and hence $J_m^+$ cannot be unknotted with a single type A or type B move. On the other hand, when $|\sigma(q_2(J_m^+))| < 6$ we have that $-6 < m < 3$. For these eight knots we have the fractions
\[
\dfrac{-357}{-50}, \dots, \dfrac{329}{48}.
\]
Applying Theorem \ref{thm:2bridge4moves} to these knots shows that none of them can be unknotted with a single 4-move. We conclude that for all $m$, $J_m^+$ cannot be unknotted with a single type B, or C move. Hence $\widetilde{u}(K_1 \# K_2) \geq 3$.
\end{proof}

\bibliography{bibliography.bib}
\bibliographystyle{alpha}
\end{document}